\documentclass[12pt]{amsart}
\usepackage{amsmath, amssymb,amscd, amsthm, tikz-cd, mathrsfs, graphicx, mathtools}
\usepackage[vcentermath]{youngtab}
\usepackage{tikz}
\usepackage{caption}

\usepackage[alphabetic,lite]{amsrefs}

\usepackage{hyperref, calligra}
\usepackage{tabu}

\newdimen\unitsize

 \newtheorem{theorem}{Theorem}[section]
 \newtheorem{lemma}[theorem]{Lemma}
 \newtheorem{corollary}[theorem]{Corollary}
 \newtheorem{prop}[theorem]{Proposition}
 
 \newtheorem*{theoremA}{Theorem A}

\newtheorem*{theorem*}{Theorem}

\theoremstyle{definition}
\newtheorem{example}[theorem]{Example}
 \newtheorem{remark}[theorem]{Remark}

\numberwithin{equation}{section}

\newcommand{\C}{\mathbb{C}}
\newcommand{\Z}{\mathbb{Z}}
\newcommand{\GL}{\operatorname{GL}}

\newcommand{\X}{\mathscr{X}}
\renewcommand{\Z}{\mathscr{Z}}

\newcommand{\codim}{\operatorname{codim}}

\newcommand{\bS}{\mathbb{S}}

\renewcommand{\a}{\underline{a}}
\renewcommand{\b}{\underline{b}}
\newcommand{\g}{\mathfrak{g}}
\newcommand{\mf}{\mathfrak}

\renewcommand{\aa}{\underline{a}}

\newcommand{\tn}{\textnormal}
\newcommand{\defi}[1]{{\upshape\sffamily #1}}

\newcommand{\ol}{\overline}
\newcommand{\lra}{\longrightarrow}
\newcommand{\mc}[1]{\mathcal{#1}}

\newcommand{\GG}{X}
\newcommand{\Sv}[1]{Z_{#1}}
\newcommand{\Sc}[1]{O_{#1}}
\newcommand{\D}{\mathcal{D}}
\newcommand{\IC}{IC}
\newcommand{\weyl}{\prescript{I}{}{W}}

\newcommand{\DD}{\mathbb{D}}

\newcommand{\lc}[1]{\mathscr{H}_{#1}}

\author{Michael Perlman}
\address{Department of Mathematics, The University of Alabama, Tuscaloosa, AL 35401}
\email{mperlman@ua.edu}

\title{Local cohomology with support in Schubert varieties}

\begin{document}
\maketitle

\begin{abstract}
This paper is concerned with local cohomology sheaves on generalized flag varieties supported in closed Schubert varieties, which carry natural structures as (mixed Hodge) $\mathcal{D}$-modules. We employ Kazhdan--Lusztig theory and Saito's theory of mixed Hodge modules to describe a general strategy to calculate the simple composition factors, Hodge filtration, and weight filtration on these modules. Our main tool is the Grothendieck--Cousin complex, introduced by Kempf, which allows us to relate the local cohomology modules in question to parabolic Verma modules over the corresponding Lie algebra. We show that this complex underlies a complex of mixed Hodge modules, and is thus endowed with Hodge and weight filtrations. As a consequence, strictness implies that computing cohomology commutes with taking associated graded with respect to both of these filtrations. We execute this strategy to calculate the composition factors and weight filtration for Schubert varieties in the  Grassmannian, in particular showing that the weight filtration is controlled by the admissible augmented Dyck patterns of Raicu--Weyman. As an application, upon restriction to the opposite big cell, we recover the  composition factors and weight filtration on local cohomology with support in generic determinantal varieties. 
\end{abstract}

\section{Introduction}

\subsection{Overview}\label{sec:over} A general problem in commutative algebra and algebraic geometry is to calculate the local cohomology sheaves $\mathscr{H}^q_Z(\mc{O}_X)$ associated to a smooth complex variety $X$ and closed subvariety $Z\subseteq X$. In recent past, a great deal of progress has been made by taking advantage of deeper structure on these sheaves, namely that of holonomic $\D_X$-modules, where $\D_X$ is the sheaf of algebraic differential operators on $X$. This strong finiteness property has facilitated numerous new calculations, especially in the presence of a group action. In particular, when $X$ has finitely many orbits under the action of a linear algebraic group and $Z$ is an orbit closure, the theory of equivariant $\D_X$-modules has provided a fruitful framework to investigate the modules $\mathscr{H}^q_Z(\mc{O}_X)$ (see \cite{categories}). A well-understood example is the case when $X$ is a space of matrices, and $Z$ is a determinantal variety (see \cite{raicu2016local, iterated, MHM}). More recently, the extra layer of Saito's theory of mixed Hodge modules \cite{saito90} has provided new insights as to how local cohomology detects and measures singularities of $Z$ (see \cite{MP1, MP4}).

Another context in which equivariant $\D_X$-modules appear is in geometric representation theory, for instance  in the study of composition factor multiplicities of Verma modules over simple Lie algebras (see \cite{htt}), which is the subject of the Kazhdan--Lusztig theorem. In particular, the Beilinson--Bernstein theorem \cite{BB} identifies, via the global sections functor, Borel equivariant $\D_X$-modules on flag varieties with highest weight modules over the corresponding Lie algebra. Here, intersection cohomology modules associated to Schubert cells correspond to irreducible representations with trivial central character, and (Verdier duals of) push-forwards of structure sheaves on Schubert cells correspond to Verma modules. In commutative algebra and algebraic geometry, we would refer to these geometric Verma modules  as the local cohomology modules $\mathscr{H}^q_{O_w}(\mathcal{O}_X)$ with support in the (locally closed) Schubert cells $O_w$. Via the theory of mixed Hodge modules, the weight filtration on $\mathscr{H}^q_{O_w}(\mathcal{O}_X)$ induces the Jantzen filtration on the corresponding Verma module \cite{jantzen}. 

The purpose of this article is to reconcile the two theories above by calculating the local cohomology modules $\mathscr{H}^q_{Z_w}(\mc{O}_X)$ with support in closed Schubert varieties $Z_w=\ol{O}_w$ in flag varieties $X=G/P$, where $G$ is a complex simple linear algebraic group and $P$ is a parabolic subgroup. Such modules are $B$-equivariant holonomic $\D_X$-modules, where $B$ is a Borel subgroup with $B\subseteq P\subseteq G$. Each local cohomology module $\mathscr{H}^q_{Z_w}(\mc{O}_X)$ is finite length over $\D_X$ and, since Schubert cells are simply connected, has composition factors among $\mathcal{L}(Z_v,X)$, where $Z_v\subseteq Z_w$ and $\mathcal{L}(Z_v,X)$ is the intersection cohomology $\D_X$-module associated to the trivial local system on the Schubert cell $O_v$. These local cohomology modules are functorially endowed with structures as mixed Hodge modules, and in particular are endowed with two increasing filtrations of geometric origin: (1) the Hodge filtration $F_{\bullet}(\mathscr{H}^q_{Z_w}(\mc{O}_X))$, an infinite filtration by coherent $\mathcal{O}_X$-modules, and (2) the weight filtration $W_{\bullet}(\mathscr{H}^q_{Z_w}(\mc{O}_X))$, a finite filtration by holonomic $\D_X$-modules. In this work, we provide a  framework to determine the Hodge and weight filtrations on $\mathscr{H}^q_{Z_w}(\mc{O}_X)$ from those on the geometric (dual) Verma modules $\mathscr{H}^q_{O_v}(\mathcal{O}_X)$. Our main result in this direction is Theorem \ref{prop:degen}, which we now summarize.

Let $W^P$ be the relative Weyl group associated to $P$, and let $\ell$ denote the length function on $W^P$. As the Schubert cells $\{O_w\}_{w\in W^P}$ give an affine paving of $X$, the \defi{Grothendieck--Cousin complex} provides a right resolution of $\mc{O}_X$ in the category of $\D_X$-modules:
$$
\mc{GC}^{\bullet}:\quad 0  \longrightarrow  \mathscr{H}^0_{O_{w_{\circ}(P)}}(\mc{O}_X)\longrightarrow \bigoplus_{\ell(v)=d-1}\mathscr{H}^1_{O_v}(\mc{O}_X) \longrightarrow \cdots \longrightarrow \bigoplus_{\ell(v)=1}\mathscr{H}^{d-1}_{O_v}(\mc{O}_X)\longrightarrow \mathscr{H}^{d}_{O_{e}}(\mc{O}_X)\longrightarrow 0,
$$
where $e$ is the identity and $w_{\circ}(P)$ is the longest element in $W^P$, of length $d=d_X$. Therefore, for a Schubert variety $Z_w$, we have $\mathscr{H}^q_{Z_w}(\mc{O}_X)=\mathscr{H}^q_{Z_w}(\mc{GC}^{\bullet})$ for $q\geq 0$. Furthermore, using affineness of the inclusions of the Schubert cells, it can be shown that, for $0\leq q\leq d$, we have (see Theorem \ref{GCaComplex}):
\begin{equation}
\mathscr{H}^0_{Z_w}(\mc{GC}^{q})=\bigoplus_{v\leq w,\;\ell(v)=d-q}\mathscr{H}^q_{O_v}(\mc{O}_X),\quad\tn{and}\quad \mathscr{H}^q_{Z_w}(\mc{O}_X)=\mc{H}^q(\mathscr{H}^0_{Z_w}(\mc{GC}^{\bullet})).
\end{equation}
The complex $\mc{GC}^{\bullet}$ was introduced by Kempf \cite{kempf} in his study of BGG resolutions of highest weight modules over simple Lie algebras. In particular, via the Beilinson--Bernstein theorem \cite{BB}, the terms $\mathscr{H}^{d-\ell(y)}_{O_y}(\mc{O}_X)$ correspond to duals of parabolic Verma modules. Since $\Gamma(X,\mc{O}_X)=\C$, it follows that $\mc{GC}^{\bullet}$ is a geometric realization of the dual of the BGG resolution of the trivial representation. 

For $w\in W^P$, we write $\mc{GC}_w^{\bullet}=\mathscr{H}^0_{Z_w}(\mc{GC}^{\bullet})$, which by the discussion above satisfies $\mc{H}^q(\mc{GC}^{\bullet}_w)=\mathscr{H}^q_{Z_w}(\mc{O}_X)$ for $q\geq 0$. The idea to use $\mc{GC}_w^{\bullet}$ to study local cohomology is not new. For instance, Raben-Pedersen \cite{raben} uses $\mc{GC}_w^{\bullet}$ to investigate qualitative properties of $\mathscr{H}^q_{Z_w}(\mc{O}_X)$ and local cohomological dimension in the case when $X=G/B$ or $X=\operatorname{Gr}(k,n)$, though a clear picture is not obtained. The complex $\mc{GC}_w^{\bullet}$ has also been examined in positive characteristic, see \cite{KL, raben, lauritzen}. Our approach to get at the cohomology of $\mc{GC}_w^{\bullet}$ is to enhance it to a complex in the category of mixed Hodge modules, and use \textit{strictness} to obtain degeneration of the spectral sequences associated to the Hodge and weight filtrations. The following is a consequence of Theorem \ref{prop:degen}.
\begin{theoremA}
For $w\in W^P$ the complex $\mc{GC}^{\bullet}_w$ 	underlies a complex of mixed Hodge modules, and hence is endowed with a Hodge filtration $F_{\bullet}(\mc{GC}^{\bullet}_w)$ by coherent $\mc{O}_X$-modules and a weight filtration $W_{\bullet}(\mc{GC}^{\bullet}_w)$ by holonomic $\D_X$-modules. Furthermore, taking cohomology commutes with taking associated graded with respect to these filtrations, i.e. for all $p\in \mathbb{Z}$ and $q\geq 0$ we have isomorphisms
$$
\tn{gr}^F_p(\lc{Z_w}^{q}(\mc{O}_X^H))\cong \mc{H}^q(\tn{gr}^F_p(\mathcal{GC}_w^{\bullet})),\quad\tn{and}\quad\tn{gr}^W_p(\lc{Z_w}^{q}(\mc{O}_X^H))\cong \mc{H}^q(\tn{gr}^W_p(\mathcal{GC}_w^{\bullet})).
$$
\end{theoremA}

In particular, in order to calculate the composition factors and weight filtration on $\mathscr{H}^q_{Z_w}(\mc{O}_X)$, instead of working with the complicated complex $\mc{GC}_w^{\bullet}$, one can study the cohomology of the complexes $\tn{gr}^W_p(\mathcal{GC}_w^{\bullet})$, which are complexes of \textit{semi-simple} $\D_X$-modules. Furthermore, via the Beilinson--Bernstein theorem, one can first  take global sections and instead work with parabolic Verma modules and irreducible representations of the corresponding Lie algebra (see Corollary \ref{globalGC}). The composition factors and weight filtrations on the parabolic Verma modules are known, and are dictated by inverse parabolic Kazhdan--Lusztig polynomials (see \cite{CC, KT}). Recent work of Davis--Vilonen \cite[Theorem 1.1]{davis} implies that the global sections functor is also exact with respect to the Hodge filtration. 

After establishing the above framework in Sections \ref{sec:2} and \ref{sec:3}, we carry out the following:
\begin{itemize}
\item In Section \ref{sec:Grass} we investigate the weight filtration on $\mc{GC}_w^{\bullet}$ in the case when $X$ is the classical Grassmannian $\operatorname{Gr}(k,n)$. We employ Kazhdan--Lusztig theory and the known submodule lattice of the relevant parabolic Verma modules to obtain the composition factors and weight filtration on $\lc{Z_w}^{q}(\mc{O}_X^H)$. Our results in this direction are stated in 	Section \ref{sec:state}. Remarkably, the weight filtration is controlled by the admissible augmented Dyck patterns of Raicu--Weyman \cite{thick}, which also control the syzygies of determinantal thickenings \cite{amy}. In Section \ref{amyremark}, we explain the connection between local cohomology and syzygies via the BGG correspondence and super duality, showing that part of our Theorem \ref{mainthm} is equivalent to \cite[Theorem 3.7]{amy}.\
\item In Section \ref{sec:det} we restrict these calculations to the opposite big cell of the Grassmannian to recover the composition factors and weight filtration on local cohomology with support in determinantal varieties of generic matrices. This gives a new combinatorial interpretation of the composition factor formula \cite{raicu2016local} and the weight filtration formula \cite{MHM}.
\end{itemize}

At the end of Section \ref{sec:maincalc} we discuss how one might approach a similar analysis for Schubert varieties in the remaining Hermitian symmetric spaces.

\subsection{Statement of results for the Grassmannian} \label{sec:state}

For positive integers $k<n$ we consider the Grassmannian $\GG=\operatorname{Gr}(k,V)$ of $k$-dimensional subspaces of an $n$-dimensional complex vector space $V$. We fix a complete flag $\mc{V}_{\bullet}$ of vector subspaces in $V$
$$
0\subset V_1\subset \cdots \subset V_{n-1}\subset V_n=V,
$$
with $\dim V_i=i$. Given a partition $\a=(a_1,\cdots,a_k)$ satisfying
\begin{equation}\label{fitInRectangle}
n-k\geq a_1\geq a_2\geq \cdots \geq a_k\geq 0,
\end{equation}
we define the Schubert variety $\Sv{\a}\subseteq \GG$ as the closed subset
$$
\Sv{\a}=\big\{W\in \GG \mid \dim(V_{k+a_i+1-i}\cap W)\geq k+1-i\;\;\tn{for all $i$}\big\}.
$$
It is the closure of the Schubert cell $\Sc{\aa}$. This variety has dimension
$$
\dim \Sv{\a}=|\a|=a_1+\cdots +a_k,
$$
and given two partitions $\a$, $\b$ satisfying (\ref{fitInRectangle}) we have
$$
\Sv{\b}\subseteq \Sv{\a} \quad \iff \quad \b \subseteq \a,
$$
where the latter notation means that $b_i\leq a_i$ for all $1\leq i \leq k$.

We identify a partition $\a=(a_1,a_2,\cdots ,a_k)$ with its Young diagram
$$
\a=(a_1,a_2,\cdots ,a_k) \quad \leftrightarrow \quad \{(i,j)\in \mathbb{Z}^2_+\mid 1\leq i \leq k,\; 1\leq j\leq a_i\},
$$
which we picture as a left-justified collection of boxes, with $a_i$ boxes in row $i$. For example, the partition $\aa=(5,4,2,2)$ corresponds to the following Young diagram:\\

\begin{center}
\begin{minipage}{.3\textwidth}
\centering
$\underline{a}=(5,4,2,2)$	
\end{minipage} $\longleftrightarrow$
\begin{minipage}{.3\textwidth}
\centering	
\begin{tikzpicture}[x=\unitsize,y=\unitsize,baseline=0]
\tikzset{vertex/.style={}}%
\tikzset{edge/.style={  thick}}%
\draw[edge] (0,0) -- (4,0);
\draw[edge] (0,2) -- (4,2);
\draw[edge] (0,4) -- (8,4);
\draw[edge] (0,6) -- (10,6);
\draw[edge] (0,8) -- (10,8);
\draw[edge] (0,0) -- (0,8);
\draw[edge] (2,0) -- (2,8);
\draw[edge] (4,0) -- (4,8);
\draw[edge] (6,4) -- (6,8);
\draw[edge] (8,4) -- (8,8);
\draw[edge] (10,6) -- (10,8);
\end{tikzpicture}%
\end{minipage}
\end{center}

\medskip

A path $P$ in $\mathbb{Z}^2_+$ is a collection of boxes
$$
P=\{ (i_1,j_1),(i_2,j_2),\cdots ,(i_t, j_t)\},
$$
satisfying the following for all $1\leq s \leq t-1$: 
$$
(i_{s+1},j_{s+1})=(i_s-1, j_s)\quad \tn{or} \quad (i_{s+1},j_{s+1})=(i_s, j_s+1).
$$

In other words, a path is a connected collection of boxes obtained by starting at some box $(i_1,j_1)$, and walking in any combination of North and East. We write $|P|=t$ for the length of $P$, which is the number of boxes in $P$. 
The path $P$ is called a \defi{Dyck path} if for some $d\geq 0$ we have 
\begin{itemize}
\item $i_1+j_1=i_t+j_t=d$,
\item $i_s+j_s\leq d$ for all $1\leq s\leq t$.
\end{itemize}
In words, $(i_1,j_1)$ and $(i_t,j_t)$ lie on the same antidiagonal of $\mathbb{Z}^2_+$, and no intermediate box $(i_s,j_s)$ lies below this anti-diagonal. In this case, we say that $P$ is a Dyck path of level $d$, and that $(i_1,j_1)$ and $(i_t,j_t)$ lie on the $d$-th antidiagonal. The following three paths are examples of Dyck paths: \\

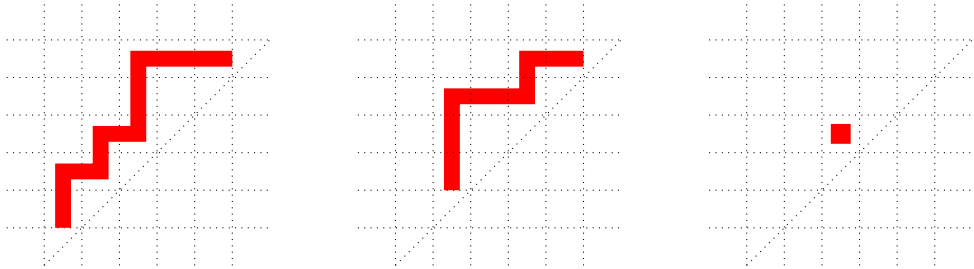
\begin{center}
\begin{minipage}{.18\textwidth}
\centering 
\begin{tikzpicture}[x=\unitsize,y=\unitsize,baseline=0]
\tikzset{vertex/.style={}}%
\tikzset{edge/.style={  thick}}%
\draw[dotted] (0,0) -- (14,0);
\draw[dotted] (0,2) -- (14,2);
\draw[dotted] (0,4) -- (14,4);
\draw[dotted] (0,6) -- (14,6);
\draw[dotted] (0,8) -- (14,8);
\draw[dotted] (0,10) -- (14,10);
\draw[dotted] (2,-2) -- (2,12);
\draw[dotted] (4,-2) -- (4,12);
\draw[dotted] (6,-2) -- (6,12);
\draw[dotted] (8,-2) -- (8,12);
\draw[dotted] (10,-2) -- (10,12);
\draw[dotted] (12,-2) -- (12,12);
\draw[dotted] (2,-2) -- (14,10);
\draw[red, line width=6pt] (3,0) -- (3,3) -- (5,3) -- (5,5) -- (7,5) -- (7,9) -- (12,9);
\end{tikzpicture}
\end{minipage}
\quad\quad\quad   
\begin{minipage}{.18\textwidth}
\centering 
\begin{tikzpicture}[x=\unitsize,y=\unitsize,baseline=0]
\tikzset{vertex/.style={}}%
\tikzset{edge/.style={  thick}}%
\draw[red, line width=6pt]  (5,2) -- (5,7) -- (9,7) -- (9, 9) -- (12,9) ;
\draw[dotted] (0,0) -- (14,0);
\draw[dotted] (0,2) -- (14,2);
\draw[dotted] (0,4) -- (14,4);
\draw[dotted] (0,6) -- (14,6);
\draw[dotted] (0,8) -- (14,8);
\draw[dotted] (0,10) -- (14,10);
\draw[dotted] (2,-2) -- (2,12);
\draw[dotted] (4,-2) -- (4,12);
\draw[dotted] (6,-2) -- (6,12);
\draw[dotted] (8,-2) -- (8,12);
\draw[dotted] (10,-2) -- (10,12);
\draw[dotted] (12,-2) -- (12,12);
\draw[dotted] (2,-2) -- (14,10);
\end{tikzpicture}
\end{minipage}
\quad\quad \quad 
\begin{minipage}{.18\textwidth}
\centering 
\begin{tikzpicture}[x=\unitsize,y=\unitsize,baseline=0]
\tikzset{vertex/.style={}}%
\tikzset{edge/.style={  thick}}%
\draw[dotted] (0,0) -- (14,0);
\draw[dotted] (0,2) -- (14,2);
\draw[dotted] (0,4) -- (14,4);
\draw[dotted] (0,6) -- (14,6);
\draw[dotted] (0,8) -- (14,8);
\draw[dotted] (0,10) -- (14,10);
\draw[dotted] (2,-2) -- (2,12);
\draw[dotted] (4,-2) -- (4,12);
\draw[dotted] (6,-2) -- (6,12);
\draw[dotted] (8,-2) -- (8,12);
\draw[dotted] (10,-2) -- (10,12);
\draw[dotted] (12,-2) -- (12,12);
\draw[color=red,fill=red] (6.5,4.5) rectangle (7.5,5.5);
\draw[dotted] (2,-2) -- (14,10);
\end{tikzpicture}
\end{minipage}
\end{center}

\medskip

The next two paths fail to be Dyck paths:\\
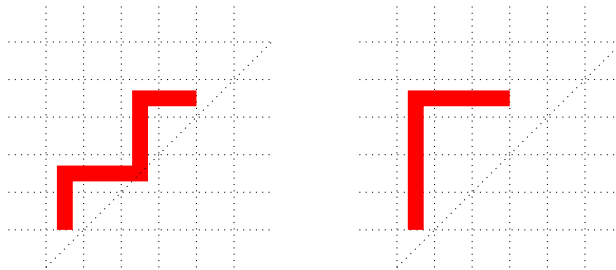
\begin{center}
\begin{minipage}{.18\textwidth}
\centering 
\begin{tikzpicture}[x=\unitsize,y=\unitsize,baseline=0]
\tikzset{vertex/.style={}}%
\tikzset{edge/.style={  thick}}%
\draw[dotted] (0,0) -- (14,0);
\draw[dotted] (0,2) -- (14,2);
\draw[dotted] (0,4) -- (14,4);
\draw[dotted] (0,6) -- (14,6);
\draw[dotted] (0,8) -- (14,8);
\draw[dotted] (0,10) -- (14,10);
\draw[dotted] (2,-2) -- (2,12);
\draw[dotted] (4,-2) -- (4,12);
\draw[dotted] (6,-2) -- (6,12);
\draw[dotted] (8,-2) -- (8,12);
\draw[dotted] (10,-2) -- (10,12);
\draw[dotted] (12,-2) -- (12,12);
\draw[red, line width=6pt] (3,0) -- (3,3) -- (7,3) -- (7,7) -- (10,7) ;
\draw[dotted] (2,-2) -- (14,10);
\end{tikzpicture}
\end{minipage}
\quad\quad\quad
\begin{minipage}{.18\textwidth}
\centering 
\begin{tikzpicture}[x=\unitsize,y=\unitsize,baseline=0]
\tikzset{vertex/.style={}}%
\tikzset{edge/.style={  thick}}%
\draw[dotted] (0,0) -- (14,0);
\draw[dotted] (0,2) -- (14,2);
\draw[dotted] (0,4) -- (14,4);
\draw[dotted] (0,6) -- (14,6);
\draw[dotted] (0,8) -- (14,8);
\draw[dotted] (0,10) -- (14,10);
\draw[dotted] (2,-2) -- (2,12);
\draw[dotted] (4,-2) -- (4,12);
\draw[dotted] (6,-2) -- (6,12);
\draw[dotted] (8,-2) -- (8,12);
\draw[dotted] (10,-2) -- (10,12);
\draw[dotted] (12,-2) -- (12,12);
\draw[red, line width=6pt] (3,0) -- (3,7) -- (8,7)  ;
\draw[dotted] (2,-2) -- (14,10);
\end{tikzpicture}
\end{minipage}
\end{center}
\medskip

The left path fails the second condition of a Dyck path, and the right path fails the first condition.

Observe that the properties of a Dyck path $P$ imply that $|P|$ is odd, $(i_2,j_2)=(i_1-1,j_1)$, and $(i_t,j_t)=(i_{t-1},j_{t-1}+1)$, i.e. a Dyck path starts by moving North and ends by moving East.

Following Raicu--Weyman \cite{thick} we define an \defi{augmented Dyck path}  as a pair $\tilde{P}=(P,B)$ where $P$ is a Dyck path and $B$ is a set of boxes, called the \defi{bullets} in $\tilde{P}$. The set $B$ is a disjoint union of two sets of boxes $B = B_{head} \sqcup B_{tail}$, where 
\begin{equation}\label{defBullets}
\begin{aligned}
B_{head}= \{(i_t,j_t+1), (i_t,j_t+2),\cdots,(i_t,j_t+u)\} &\mbox{ for some }u\geq 0,\mbox{ and} \\
B_{tail} = \{(i_1+1,j_1), (i_1+2,j_1), \cdots,(i_1+v,j_1)\} &\mbox{ for some }v\geq 0. \\
\end{aligned}
\end{equation}
In words, $B_{head}$ is a connected collection of boxes obtained from starting in the box directly East of $(i_t,j_t)$, and taking some number of steps East. Similarly, $B_{tail}$ is a connected collection of boxes obtained from starting in the box directly South of $(i_1,j_1)$, and taking some number of steps South.\\

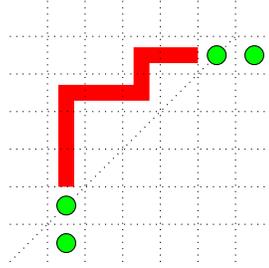
\begin{center}
\begin{minipage}{.18\textwidth}
\centering 
\begin{tikzpicture}[x=\unitsize,y=\unitsize,baseline=0]
\tikzset{vertex/.style={}}%
\tikzset{edge/.style={  thick}}%
\draw[red, line width=6pt]  (3,2) -- (3,7) -- (7,7) -- (7, 9) -- (10,9) ;
\draw[dotted] (0,0) -- (14,0);
\draw[dotted] (0,2) -- (14,2);
\draw[dotted] (0,4) -- (14,4);
\draw[dotted] (0,6) -- (14,6);
\draw[dotted] (0,8) -- (14,8);
\draw[dotted] (0,10) -- (14,10);
\draw[dotted] (2,-2) -- (2,12);
\draw[dotted] (4,-2) -- (4,12);
\draw[dotted] (6,-2) -- (6,12);
\draw[dotted] (8,-2) -- (8,12);
\draw[dotted] (10,-2) -- (10,12);
\draw[dotted] (12,-2) -- (12,12);
\draw[dotted] (0,-2) -- (12,10);
\draw[fill=green] (11,9) circle [radius=0.5];
\draw[fill=green] (13,9) circle [radius=0.5];
\draw[fill=green] (3,1) circle [radius=0.5];
\draw[fill=green] (3,-1) circle [radius=0.5];
\end{tikzpicture}
\end{minipage}
\end{center}
\medskip

Given a partition $\a$, an \defi{augmented Dyck pattern}  in $\aa$ is a collection 
$$
\mathbb{D} = (D_1,D_2,\cdots,D_r;\mathbb{B}),
$$
 where 
\begin{itemize}
 \item each $D_i\subseteq \aa$ is a Dyck path and $\mathbb{B} \subseteq \aa$ is a finite set of boxes;
 \item the sets $D_1,D_2,\cdots,D_r$ and $\mathbb{B}$ are pairwise disjoint;
 \item $\mathbb{B}$ can be expressed as a union
 \begin{equation}\label{bulletUnion}
 \mathbb{B} = B_1 \cup B_2 \cup \cdots \cup B_r
 \end{equation}
 in such a way that $(D_i,B_i)$ is an augmented Dyck path for every $i=1,\cdots,r$.
\end{itemize}

\noindent We emphasize that the expression (\ref{bulletUnion}) need not be unique, and the sets $B_i$ are not necessarily disjoint. We write $|\mathbb{D}|=r$ for the number of Dyck paths in $\mathbb{D}$ (so we are not counting elements of $\mathbb{B}$). We often refer to augmented Dyck patterns simply as ``Dyck patterns". 

 We define the support of $\mathbb{D}$ by
\begin{equation}\label{eq:def-supp}
 \operatorname{supp}(\mathbb{D}) = D_1 \cup D_2 \cup \cdots \cup D_r \cup \mathbb{B}.
\end{equation}
If $\aa$ is a partition and $\DD$ is an augmented Dyck pattern in $\aa$, we write 
\begin{equation}\label{defDa}
\a^{\mathbb{D}}=\aa \setminus \operatorname{supp}(\mathbb{D}).
\end{equation}

Raicu--Weyman \cite{thick} define a Dyck pattern $\DD=(D_1,D_2,\cdots,D_r;\mathbb{B})$ in $\aa$ to be \defi{admissible} if the following conditions are satisfied:
\begin{enumerate}
\item $a^{\mathbb{D}}$ is a set of boxes corresponding to a partition,
\item (Covering Condition) for every $p\neq q$, if there exists a box $(i',j')\in D_q$ which is located directly North, Northwest, or West from a box $(i,j)\in D_p$, then every box located directly North, Northwest, or West from a box in $D_p$ must also belong to $D_p$ or $D_q$,
\item there is no box in $\mathbb{B}$ which is located directly North, Northwest, or West from a box in any Dyck path $D_p$.	
\end{enumerate}

The following are four examples of admissible Dyck patterns in $\a=(5,4,3,2,2)$:\\

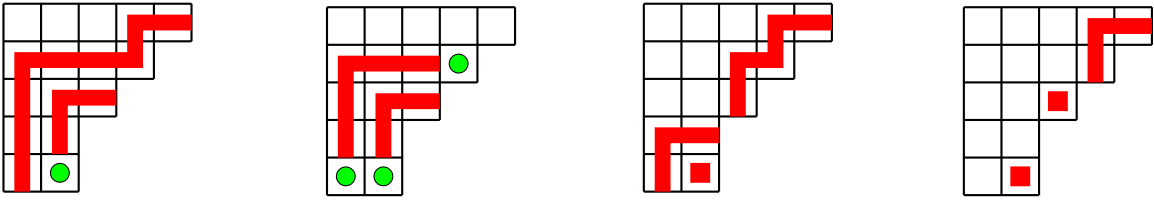
\begin{center}
\begin{minipage}{.18\textwidth}
\centering 
\begin{tikzpicture}[x=\unitsize,y=\unitsize,baseline=0]
\tikzset{vertex/.style={}}%
\tikzset{edge/.style={thick}}%
\draw[edge] (0,0) -- (4,0);
\draw[edge] (0,2) -- (4,2);
\draw[edge] (0,4) -- (6,4);
\draw[edge] (0,6) -- (8,6);
\draw[edge] (0,8) -- (10,8);
\draw[edge] (0,10) -- (10,10);
\draw[edge] (0,0) -- (0,10);
\draw[edge] (2,0) -- (2,10);
\draw[edge] (4,0) -- (4,10);
\draw[edge] (6,4) -- (6,10);
\draw[edge] (8,6) -- (8,10);
\draw[edge] (10,8) -- (10,10);
\draw[red, line width=6pt] (3,2) -- (3,5) --(6,5) ;
\draw[red, line width=6pt] (1,0) -- (1,7) --(7,7) -- (7,9) -- (10,9) ;
\draw[fill=green] (3,1) circle [radius=0.5];
\end{tikzpicture}
\end{minipage}
\quad\quad
\begin{minipage}{.18\textwidth}
\centering 
\begin{tikzpicture}[x=\unitsize,y=\unitsize,baseline=0]
\tikzset{vertex/.style={}}%
\tikzset{edge/.style={thick}}%
\draw[edge] (0,0) -- (4,0);
\draw[edge] (0,2) -- (4,2);
\draw[edge] (0,4) -- (6,4);
\draw[edge] (0,6) -- (8,6);
\draw[edge] (0,8) -- (10,8);
\draw[edge] (0,10) -- (10,10);
\draw[edge] (0,0) -- (0,10);
\draw[edge] (2,0) -- (2,10);
\draw[edge] (4,0) -- (4,10);
\draw[edge] (6,4) -- (6,10);
\draw[edge] (8,6) -- (8,10);
\draw[edge] (10,8) -- (10,10);
\draw[red, line width=6pt] (3,2) -- (3,5) --(6,5) ;
\draw[red, line width=6pt] (1,2) -- (1,7) --(6,7) ;
\draw[fill=green] (7,7) circle [radius=0.5];
\draw[fill=green] (3,1) circle [radius=0.5];
\draw[fill=green] (1,1) circle [radius=0.5];
\end{tikzpicture}
\end{minipage}
\quad\quad
\begin{minipage}{.18\textwidth}
\centering 
\begin{tikzpicture}[x=\unitsize,y=\unitsize,baseline=0]
\tikzset{vertex/.style={}}%
\tikzset{edge/.style={  thick}}%
\draw[edge] (0,0) -- (4,0);
\draw[edge] (0,2) -- (4,2);
\draw[edge] (0,4) -- (6,4);
\draw[edge] (0,6) -- (8,6);
\draw[edge] (0,8) -- (10,8);
\draw[edge] (0,10) -- (10,10);
\draw[edge] (0,0) -- (0,10);
\draw[edge] (2,0) -- (2,10);
\draw[edge] (4,0) -- (4,10);
\draw[edge] (6,4) -- (6,10);
\draw[edge] (8,6) -- (8,10);
\draw[edge] (10,8) -- (10,10);
\draw[red, line width=6pt] (1,0) -- (1,3) --(4,3) ;
\draw[color=red,fill=red] (2.5,0.5) rectangle (3.5,1.5);
\draw[red, line width=6pt] (5,4) -- (5,7) --(7,7) -- (7,9) -- (10,9) ;
\end{tikzpicture}
\end{minipage}
\quad\quad
\begin{minipage}{.18\textwidth}
\centering 
\begin{tikzpicture}[x=\unitsize,y=\unitsize,baseline=0]
\tikzset{vertex/.style={}}%
\tikzset{edge/.style={  thick}}%
\draw[edge] (0,0) -- (4,0);
\draw[edge] (0,2) -- (4,2);
\draw[edge] (0,4) -- (6,4);
\draw[edge] (0,6) -- (8,6);
\draw[edge] (0,8) -- (10,8);
\draw[edge] (0,10) -- (10,10);
\draw[edge] (0,0) -- (0,10);
\draw[edge] (2,0) -- (2,10);
\draw[edge] (4,0) -- (4,10);
\draw[edge] (6,4) -- (6,10);
\draw[edge] (8,6) -- (8,10);
\draw[edge] (10,8) -- (10,10);
\draw[red, line width=6pt] (7,6) -- (7,9) --(10,9) ;
\draw[color=red,fill=red] (2.5,0.5) rectangle (3.5,1.5);
\draw[color=red,fill=red] (4.5,4.5) rectangle (5.5,5.5);
\end{tikzpicture}
\end{minipage}
\end{center}
\medskip

The following three patterns are not admissible:\\

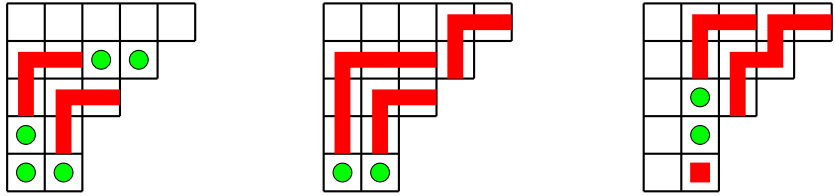
\begin{center}
\begin{minipage}{.18\textwidth}
\centering 
\begin{tikzpicture}[x=\unitsize,y=\unitsize,baseline=0]
\tikzset{vertex/.style={}}%
\tikzset{edge/.style={  thick}}%
\draw[edge] (0,0) -- (4,0);
\draw[edge] (0,2) -- (4,2);
\draw[edge] (0,4) -- (6,4);
\draw[edge] (0,6) -- (8,6);
\draw[edge] (0,8) -- (10,8);
\draw[edge] (0,10) -- (10,10);
\draw[edge] (0,0) -- (0,10);
\draw[edge] (2,0) -- (2,10);
\draw[edge] (4,0) -- (4,10);
\draw[edge] (6,4) -- (6,10);
\draw[edge] (8,6) -- (8,10);
\draw[edge] (10,8) -- (10,10);
\draw[red, line width=6pt] (3,2) -- (3,5) --(6,5) ;
\draw[fill=green] (3,1) circle [radius=0.5];
\draw[red, line width=6pt] (1,4) -- (1,7) --(4,7) ;
\draw[fill=green] (1,1) circle [radius=0.5];
\draw[fill=green] (1,3) circle [radius=0.5];
\draw[fill=green] (7,7) circle [radius=0.5];
\draw[fill=green] (5,7) circle [radius=0.5];
\end{tikzpicture}
\end{minipage}
\quad\quad
\begin{minipage}{.18\textwidth}
\centering 
\begin{tikzpicture}[x=\unitsize,y=\unitsize,baseline=0]
\tikzset{vertex/.style={}}%
\tikzset{edge/.style={  thick}}%
\draw[edge] (0,0) -- (4,0);
\draw[edge] (0,2) -- (4,2);
\draw[edge] (0,4) -- (6,4);
\draw[edge] (0,6) -- (8,6);
\draw[edge] (0,8) -- (10,8);
\draw[edge] (0,10) -- (10,10);
\draw[edge] (0,0) -- (0,10);
\draw[edge] (2,0) -- (2,10);
\draw[edge] (4,0) -- (4,10);
\draw[edge] (6,4) -- (6,10);
\draw[edge] (8,6) -- (8,10);
\draw[edge] (10,8) -- (10,10);
\draw[red, line width=6pt] (3,2) -- (3,5) --(6,5) ;
\draw[red, line width=6pt] (1,2) -- (1,7) --(6,7) ;
\draw[red, line width=6pt] (7,6) -- (7,9) --(10,9) ;
\draw[fill=green] (3,1) circle [radius=0.5];
\draw[fill=green] (1,1) circle [radius=0.5];
\end{tikzpicture}
\end{minipage}
\quad\quad
\begin{minipage}{.18\textwidth}
\centering 
\begin{tikzpicture}[x=\unitsize,y=\unitsize,baseline=0]
\tikzset{vertex/.style={}}%
\tikzset{edge/.style={  thick}}%
\draw[edge] (0,0) -- (4,0);
\draw[edge] (0,2) -- (4,2);
\draw[edge] (0,4) -- (6,4);
\draw[edge] (0,6) -- (8,6);
\draw[edge] (0,8) -- (10,8);
\draw[edge] (0,10) -- (10,10);
\draw[edge] (0,0) -- (0,10);
\draw[edge] (2,0) -- (2,10);
\draw[edge] (4,0) -- (4,10);
\draw[edge] (6,4) -- (6,10);
\draw[edge] (8,6) -- (8,10);
\draw[edge] (10,8) -- (10,10);
\draw[red, line width=6pt] (5,4) -- (5,7) --(7,7) -- (7,9) -- (10,9) ;
\draw[red, line width=6pt] (3,6) -- (3,9) -- (6,9) ;
\draw[fill=green] (3,3) circle [radius=0.5];
\draw[fill=green] (3,5) circle [radius=0.5];
\draw[color=red,fill=red] (2.5,0.5) rectangle (3.5,1.5);
\end{tikzpicture}
\end{minipage}

\end{center}

\medskip

The covering condition is Rule II in \cite[Section 3.1]{path}. This condition may also be described in terms of boundary strip removals from skew diagrams \cite[Section 2]{brenti2}.

We distinguish between Dyck patterns with bullets and those without bullets as follows:
\begin{align}
& Dyck^{\bullet}(\a)=\{\textnormal{admissible augmented Dyck patterns in $\a$}\},\\
& Dyck(\a)=\{\DD=(D_1,D_2,\cdots,D_r;\mathbb{B})\in Dyck^{\bullet}(\a)\mid \mathbb{B}=\emptyset\}.	
\end{align}
We write $\mathbb{D}=(D_1,\cdots,D_r)$ when $\mathbb{B}=\emptyset$. The empty pattern $\mathbb{D}=\emptyset$ satisfies $\a^{\mathbb{D}}=\a$ and belongs to both $Dyck^{\bullet}(\a)$ and $Dyck(\a)$.

With these combinatorial definitions, we now prepare to state the main result. Let $d_X=\dim X=k(n-k)$. Given a partition $\aa$ and $q\in \mathbb{Z}_{\geq 0}$ we consider the following sets of Dyck patterns:
\begin{equation}\label{eq:patternslc}
\mc{A}(\aa; q)=\{ \mathbb{D} = (D_1,D_2,\cdots,D_r;\mathbb{B})\in Dyck^{\bullet}(\a) \mid |D_i|\geq 3\; \tn{for $i=1,\cdots ,r$},\; |\mathbb{B}|=q+|\a|-d_{X}\}.
\end{equation}
For $p\in \mathbb{Z}_{\geq 0}$ we consider the subset of $\mc{A}(\aa; q)$ consisting of patterns with $p-q-d_X$ paths:
\begin{equation}
\mathcal{A}_p(\aa;q)=\{\mathbb{D}\in \mc{A}(\aa; q)\mid |\mathbb{D}|=p-q-d_X\}.
\end{equation}
For the following statement, we write $\mc{L}(\b)=\mc{L}(\Sv{\b},X)$ for the intersection cohomology $\D_X$-module associated to the trivial local system on the Schubert cell $\Sc{\b}$. For $k\in \mathbb{Z}$, the Hodge modules $IC^H_{\Sv{\b}}(k)$ provide a complete list of polarizable pure Hodge modules that may overlie $\mc{L}(\b)$ (see Section \ref{sec:D}). By the Beilinson--Bernstein theorem (see Section \ref{sec:loc}), the global sections functor is exact, and $\Gamma(X,\mc{L}(\b))$ is the irreducible $\mf{sl}_n(\C)$-representation $L(\b)$, whose highest weight is described in Section \ref{sec:young}.

\begin{theorem}\label{mainthm}
With the notation above, we have the following expression in the Grothendieck group of $\D_{\GG}$-modules (resp. representations of $\mathfrak{sl}_{n}(\C))$ for $q\geq 0$:
$$
\big[ \lc{Z_{\a}}^{q}(\mc{O}_{\GG})\big]=\sum_{\mathbb{D}\in \mc{A}(\a;q)} \left[\mathcal{L}(\a^{\mathbb{D}})\right],\quad\tn{and}\quad \big[H^q_{Z_{\a}}(\GG,\mc{O}_{\GG})\big]=\sum_{\mathbb{D}\in \mc{A}(\a;q)} \left[L(\a^{\mathbb{D}})\right].
$$
Furthermore, the following is true about the weight filtration $W_{\bullet}$ for $p\in \mathbb{Z}$:
$$
\tn{gr}^W_p\lc{Z_{\a}}^q(\mc{O}^H_{\GG})=\bigoplus_{\mathbb{D}\in \mathcal{A}_p(\aa;q)} IC^H_{Z_{\a^{\mathbb{D}}}}\left(\frac{\left(\left|\a^{\mathbb{D}}\right|-p\right)}{2}\right).
$$
	
\end{theorem}

\noindent In Section \ref{amyremark}, we explain a remarkable connection to syzygies of some equivariant ideals. Specifically, via super duality, an equivalence of categories between representations of Lie (super)algebras, our description of $[H^q_{Z_{\a}}(\GG,\mc{O}_{\GG})]$ is equivalent to Huang's theorem on linear strands of syzygies of determinantal thickenings \cite{amy}. 

We carry out a first example.

\begin{example}\label{ex1}
Let $X=\operatorname{Gr}(4,9)$ and $\underline{a}=(5,4,2,2)$. There are eight elements of $Dyck^{\bullet}(\a)$ for which all paths have length at least three:\\

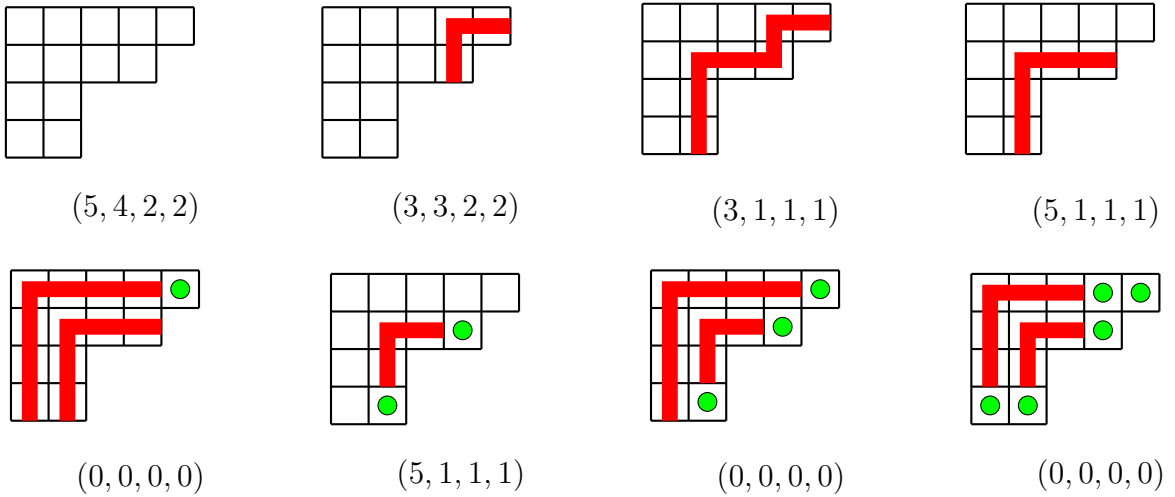
\begin{center}
\begin{minipage}{.18\textwidth}
\centering 
\begin{tikzpicture}[x=\unitsize,y=\unitsize,baseline=0]
\tikzset{vertex/.style={}}%
\tikzset{edge/.style={  thick}}%
\draw[edge] (0,0) -- (4,0);
\draw[edge] (0,2) -- (4,2);
\draw[edge] (0,4) -- (8,4);
\draw[edge] (0,6) -- (10,6);
\draw[edge] (0,8) -- (10,8);
\draw[edge] (0,0) -- (0,8);
\draw[edge] (2,0) -- (2,8);
\draw[edge] (4,0) -- (4,8);
\draw[edge] (6,4) -- (6,8);
\draw[edge] (8,4) -- (8,8);
\draw[edge] (10,6) -- (10,8);
\end{tikzpicture}
\captionsetup{labelformat=empty}	
\captionof{figure}{$(5,4,2,2)$}	
\end{minipage}
\quad\quad
\begin{minipage}{.18\textwidth}
\centering 
\begin{tikzpicture}[x=\unitsize,y=\unitsize,baseline=0]
\tikzset{vertex/.style={}}%
\tikzset{edge/.style={  thick}}%
\draw[edge] (0,0) -- (4,0);
\draw[edge] (0,2) -- (4,2);
\draw[edge] (0,4) -- (8,4);
\draw[edge] (0,6) -- (10,6);
\draw[edge] (0,8) -- (10,8);
\draw[edge] (0,0) -- (0,8);
\draw[edge] (2,0) -- (2,8);
\draw[edge] (4,0) -- (4,8);
\draw[edge] (6,4) -- (6,8);
\draw[edge] (8,4) -- (8,8);
\draw[edge] (10,6) -- (10,8);
\draw[red, line width=6pt] (7,4) -- (7,7) -- (10,7) ;
\end{tikzpicture}
\captionsetup{labelformat=empty}	
\captionof{figure}{$(3,3,2,2)$}
\end{minipage}
\quad\quad
\begin{minipage}{.18\textwidth}
\centering 
\begin{tikzpicture}[x=\unitsize,y=\unitsize,baseline=0]
\tikzset{vertex/.style={}}%
\tikzset{edge/.style={  thick}}%
\draw[edge] (0,0) -- (4,0);
\draw[edge] (0,2) -- (4,2);
\draw[edge] (0,4) -- (8,4);
\draw[edge] (0,6) -- (10,6);
\draw[edge] (0,8) -- (10,8);
\draw[edge] (0,0) -- (0,8);
\draw[edge] (2,0) -- (2,8);
\draw[edge] (4,0) -- (4,8);
\draw[edge] (6,4) -- (6,8);
\draw[edge] (8,4) -- (8,8);
\draw[edge] (10,6) -- (10,8);
\draw[red, line width=6pt] (3,0) -- (3,5) --(7,5) -- (7,7) -- (10,7) ;
\end{tikzpicture}
\captionsetup{labelformat=empty}	
\captionof{figure}{$(3,1,1,1)$}
\end{minipage}
\quad\quad
\begin{minipage}{.18\textwidth}
\centering 
\begin{tikzpicture}[x=\unitsize,y=\unitsize,baseline=0]
\tikzset{vertex/.style={}}%
\tikzset{edge/.style={  thick}}%
\draw[edge] (0,0) -- (4,0);
\draw[edge] (0,2) -- (4,2);
\draw[edge] (0,4) -- (8,4);
\draw[edge] (0,6) -- (10,6);
\draw[edge] (0,8) -- (10,8);
\draw[edge] (0,0) -- (0,8);
\draw[edge] (2,0) -- (2,8);
\draw[edge] (4,0) -- (4,8);
\draw[edge] (6,4) -- (6,8);
\draw[edge] (8,4) -- (8,8);
\draw[edge] (10,6) -- (10,8);
\draw[red, line width=6pt] (3,0) -- (3,5) --(8,5) ;
\end{tikzpicture}
\captionsetup{labelformat=empty}	
\captionof{figure}{$(5,1,1,1)$}
\end{minipage}
\medskip
\medskip

\begin{minipage}{.18\textwidth}
\centering 
\begin{tikzpicture}[x=\unitsize,y=\unitsize,baseline=0]
\tikzset{vertex/.style={}}%
\tikzset{edge/.style={  thick}}%
\draw[edge] (0,0) -- (4,0);
\draw[edge] (0,2) -- (4,2);
\draw[edge] (0,4) -- (8,4);
\draw[edge] (0,6) -- (10,6);
\draw[edge] (0,8) -- (10,8);
\draw[edge] (0,0) -- (0,8);
\draw[edge] (2,0) -- (2,8);
\draw[edge] (4,0) -- (4,8);
\draw[edge] (6,4) -- (6,8);
\draw[edge] (8,4) -- (8,8);
\draw[edge] (10,6) -- (10,8);
\draw[red, line width=6pt] (3,0) -- (3,5) --(8,5) ;
\draw[red, line width=6pt] (1,0) -- (1,7) --(8,7) ;
\draw[fill=green] (9,7) circle [radius=0.5];
\end{tikzpicture}
\captionsetup{labelformat=empty}	
\captionof{figure}{$(0,0,0,0)$}	
\end{minipage}
\quad\quad
\begin{minipage}{.18\textwidth}
\centering 
\begin{tikzpicture}[x=\unitsize,y=\unitsize,baseline=0]
\tikzset{vertex/.style={}}%
\tikzset{edge/.style={  thick}}%
\draw[edge] (0,0) -- (4,0);
\draw[edge] (0,2) -- (4,2);
\draw[edge] (0,4) -- (8,4);
\draw[edge] (0,6) -- (10,6);
\draw[edge] (0,8) -- (10,8);
\draw[edge] (0,0) -- (0,8);
\draw[edge] (2,0) -- (2,8);
\draw[edge] (4,0) -- (4,8);
\draw[edge] (6,4) -- (6,8);
\draw[edge] (8,4) -- (8,8);
\draw[edge] (10,6) -- (10,8);
\draw[red, line width=6pt] (3,2) -- (3,5) --(6,5) ;
\draw[fill=green] (7,5) circle [radius=0.5];
\draw[fill=green] (3,1) circle [radius=0.5];
\end{tikzpicture}
\captionsetup{labelformat=empty}	
\captionof{figure}{$(5,1,1,1)$}
\end{minipage}
\quad\quad
\begin{minipage}{.18\textwidth}
\centering 
\begin{tikzpicture}[x=\unitsize,y=\unitsize,baseline=0]
\tikzset{vertex/.style={}}%
\tikzset{edge/.style={  thick}}%
\draw[edge] (0,0) -- (4,0);
\draw[edge] (0,2) -- (4,2);
\draw[edge] (0,4) -- (8,4);
\draw[edge] (0,6) -- (10,6);
\draw[edge] (0,8) -- (10,8);
\draw[edge] (0,0) -- (0,8);
\draw[edge] (2,0) -- (2,8);
\draw[edge] (4,0) -- (4,8);
\draw[edge] (6,4) -- (6,8);
\draw[edge] (8,4) -- (8,8);
\draw[edge] (10,6) -- (10,8);
\draw[red, line width=6pt] (3,2) -- (3,5) --(6,5) ;
\draw[red, line width=6pt] (1,0) -- (1,7) --(8,7) ;
\draw[fill=green] (9,7) circle [radius=0.5];
\draw[fill=green] (7,5) circle [radius=0.5];
\draw[fill=green] (3,1) circle [radius=0.5];
\end{tikzpicture}
\captionsetup{labelformat=empty}	
\captionof{figure}{$(0,0,0,0)$}
\end{minipage}
\quad\quad
\begin{minipage}{.18\textwidth}
\centering 
\begin{tikzpicture}[x=\unitsize,y=\unitsize,baseline=0]
\tikzset{vertex/.style={}}%
\tikzset{edge/.style={  thick}}%
\draw[edge] (0,0) -- (4,0);
\draw[edge] (0,2) -- (4,2);
\draw[edge] (0,4) -- (8,4);
\draw[edge] (0,6) -- (10,6);
\draw[edge] (0,8) -- (10,8);
\draw[edge] (0,0) -- (0,8);
\draw[edge] (2,0) -- (2,8);
\draw[edge] (4,0) -- (4,8);
\draw[edge] (6,4) -- (6,8);
\draw[edge] (8,4) -- (8,8);
\draw[edge] (10,6) -- (10,8);
\draw[red, line width=6pt] (3,2) -- (3,5) --(6,5) ;
\draw[red, line width=6pt] (1,2) -- (1,7) --(6,7) ;
\draw[fill=green] (9,7) circle [radius=0.5];
\draw[fill=green] (7,7) circle [radius=0.5];
\draw[fill=green] (7,5) circle [radius=0.5];
\draw[fill=green] (3,1) circle [radius=0.5];
\draw[fill=green] (1,1) circle [radius=0.5];
\end{tikzpicture}
\captionsetup{labelformat=empty}	
\captionof{figure}{$(0,0,0,0)$}
\end{minipage}

\end{center}

\medskip

The number of bullets determines the cohomological degree, after an initial shift by codimension, which is $c=7$ in this example. The first four patterns belong to $\mc{A}(\a;7)$, the fifth belongs to $\mc{A}(\a;8)$, the sixth belongs to $\mc{A}(\a;9)$, the seventh belongs to $\mc{A}(\a,10)$, and the eighth belongs to $\mc{A}(\a;12)$. By Theorem \ref{mainthm} we have the following in the Grothendieck group of $\D_X$-modules:	
$$
\big[\mathscr{H}^{7}_{Z_{\underline{a}}}(\mathcal{O}_{X})\big]=\big[\mathcal{L}(5,4,2,2)\big]+\big[\mathcal{L}(3,3,2,2)\big]+\big[\mathcal{L}(3,1,1,1)\big]+\big[\mathcal{L}(5,1,1,1)\big],
$$
$$
\big[\mathscr{H}^{8}_{Z_{\underline{a}}}(\mathcal{O}_{X})\big]=\big[\mathcal{L}(0)\big],\quad\big[\mathscr{H}^{9}_{Z_{\underline{a}}}(\mathcal{O}_{X})\big]=\big[\mathcal{L}(5,1,1,1)\big],\quad \big[\mathscr{H}^{10}_{Z_{\underline{a}}}(\mathcal{O}_{X})\big]=\big[\mathscr{H}^{12}_{Z_{\underline{a}}}(\mathcal{O}_{X})\big]=\big[\mathcal{L}(0)\big].
$$

The graded pieces of the weight filtration are given by the number of paths plus the number of bullets, after an initial shift of $c+d_X$, which is $27$ in this example. Thus, the weights of the simple modules above are (from left to right): $27$, $28$, $28$, $28$, $30$, $30$, $32$, $34$.\hfill\mbox{$\diamond$}
\end{example}

Every Grassmannian has a unique Schubert divisor $\Sv{\a}$, where $\a$ is obtained by removing the bottom right box from the $k\times (n-k)$ rectangle. The following example illustrates the general behavior for Schubert divisors. We see that there is only one way to remove Dyck paths of lengths $\geq 3$ such that the covering condition is satisfied, yielding only one nonzero local cohomology module.

\begin{example}\label{ex2}
Set $X=\operatorname{Gr}(4,9)$ and let $\a=(5,5,5,4)$. There are four elements of $Dyck^{\bullet}(\a)$ for which all paths have length at least three:\\

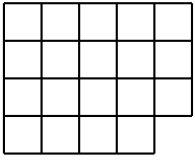
\begin{center}
\begin{minipage}{.18\textwidth}
\centering 
\begin{tikzpicture}[x=\unitsize,y=\unitsize,baseline=0]
\tikzset{vertex/.style={}}%
\tikzset{edge/.style={  thick}}%
\draw[edge] (-2,2) -- (6,2);
\draw[edge] (-2,4) -- (8,4);
\draw[edge] (-2,6) -- (8,6);
\draw[edge] (-2,8) -- (8,8);
\draw[edge] (-2,10) -- (8,10);
\draw[edge] (-2,2) -- (-2,10);
\draw[edge] (0,2) -- (0,10);
\draw[edge] (2,2) -- (2,10);
\draw[edge] (4,2) -- (4,10);
\draw[edge] (6,2) -- (6,10);
\draw[edge] (8,4) -- (8,10);
\end{tikzpicture}
\captionsetup{labelformat=empty}	
\captionof{figure}{$(5,5,5,4)$}
\end{minipage}
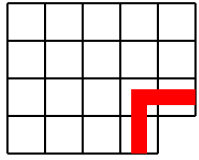
\quad\quad
\begin{minipage}{.18\textwidth}
\centering 
\begin{tikzpicture}[x=\unitsize,y=\unitsize,baseline=0]
\tikzset{vertex/.style={}}%
\tikzset{edge/.style={  thick}}%
\draw[edge] (-2,2) -- (6,2);
\draw[edge] (-2,4) -- (8,4);
\draw[edge] (-2,6) -- (8,6);
\draw[edge] (-2,8) -- (8,8);
\draw[edge] (-2,10) -- (8,10);
\draw[edge] (-2,2) -- (-2,10);
\draw[edge] (0,2) -- (0,10);
\draw[edge] (2,2) -- (2,10);
\draw[edge] (4,2) -- (4,10);
\draw[edge] (6,2) -- (6,10);
\draw[edge] (8,4) -- (8,10);
\draw[red, line width=6pt]  (5,2) -- (5,5) -- (8,5);
\end{tikzpicture}
\captionsetup{labelformat=empty}	
\captionof{figure}{$(5,5,3,3)$}
\end{minipage}
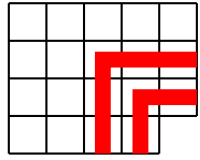
\quad\quad
\begin{minipage}{.18\textwidth}
\centering 
\begin{tikzpicture}[x=\unitsize,y=\unitsize,baseline=0]
\tikzset{vertex/.style={}}%
\tikzset{edge/.style={  thick}}%
\draw[edge] (-2,2) -- (6,2);
\draw[edge] (-2,4) -- (8,4);
\draw[edge] (-2,6) -- (8,6);
\draw[edge] (-2,8) -- (8,8);
\draw[edge] (-2,10) -- (8,10);
\draw[edge] (-2,2) -- (-2,10);
\draw[edge] (0,2) -- (0,10);
\draw[edge] (2,2) -- (2,10);
\draw[edge] (4,2) -- (4,10);
\draw[edge] (6,2) -- (6,10);
\draw[edge] (8,4) -- (8,10);
\draw[red, line width=6pt]  (5,2) -- (5,5) -- (8,5);
\draw[red, line width=6pt]  (3,2) -- (3,7) -- (8,7);
\end{tikzpicture}
\captionsetup{labelformat=empty}	
\captionof{figure}{$(5,2,2,2)$}
\end{minipage}
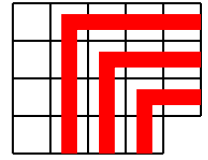
\quad\quad
\begin{minipage}{.18\textwidth}
\centering 
\begin{tikzpicture}[x=\unitsize,y=\unitsize,baseline=0]
\tikzset{vertex/.style={}}%
\tikzset{edge/.style={  thick}}%
\draw[edge] (-2,2) -- (6,2);
\draw[edge] (-2,4) -- (8,4);
\draw[edge] (-2,6) -- (8,6);
\draw[edge] (-2,8) -- (8,8);
\draw[edge] (-2,10) -- (8,10);
\draw[edge] (-2,2) -- (-2,10);
\draw[edge] (0,2) -- (0,10);
\draw[edge] (2,2) -- (2,10);
\draw[edge] (4,2) -- (4,10);
\draw[edge] (6,2) -- (6,10);
\draw[edge] (8,4) -- (8,10);
\draw[red, line width=6pt]  (5,2) -- (5,5) -- (8,5);
\draw[red, line width=6pt]  (3,2) -- (3,7) -- (8,7);
\draw[red, line width=6pt]  (1,2) -- (1,9) -- (8,9);
\end{tikzpicture}
\captionsetup{labelformat=empty}	
\captionof{figure}{$(1,1,1,1)$}
\end{minipage}

\end{center}

\medskip

There is one local cohomology module, given as follows in the Grothendieck group of $\D_X$-modules:
$$
\big[\mathscr{H}^{1}_{Z_{\a}}(\mathcal{O}_{X})\big]=\big[\mathcal{L}(5,5,5,4)\big]+\big[\mathcal{L}(5,5,3,3)\big]+\big[\mathcal{L}(5,2,2,2)\big]+\big[\mathcal{L}(1,1,1,1)\big].
$$		
The weight filtration is given by the number of paths in the pattern, after an initial shift by $c+d_X=21$. From left to right, the weights of the composition factors are: $21$, $22$, $23$, $24$.\hfill\mbox{$\diamond$}	
\end{example}

In the next example, we discuss the singular locus of $Z_{\a}$ in relation to local cohomology.

\begin{example}
Let $\a$ be a partition, and write $\a=(\tilde{a}_1^{d_1}, \tilde{a}_2^{d_2},\cdots, \tilde{a}_t^{d_t})$	with $\tilde{a}_1>\tilde{a}_2>\cdots>\tilde{a}_t$ and $d_i> 0$ for $1\leq i\leq t$. Pictorially, this expression for $\a$ decomposes the Young diagram for $\a$ into blocks of distinct widths. For instance $\a=(5^1,4^1,2^2)$ in Example \ref{ex1}, and $\a=(5^3,4^1)$ in Example \ref{ex2}. For $1\leq i\leq t-1$ we define the partitions
$$
\underline{a}^i:=(\tilde{a}_1^{d_1}, \cdots, \tilde{a}_{i-1}^{d_{i-1}},\tilde{a}_i^{d_i-1}, (\tilde{a}_{i+1}-1)^{d_{i+1}+1},\tilde{a}_{i+2}^{d_{i+2}},\cdots, \tilde{a}_t^{d_t}).
$$
In words, $\a^i$ is obtained from $\a$ by removing $\tilde{a}_i-\tilde{a}_{i+1}+1$ many boxes from the last row of the $\tilde{a}_i$ block, and removing one box from every row of the $\tilde{a}_{i+1}$ block. Together, the boxes being removed to make $\a^i$ form a path $P_i$ in the shape of a hook around the corner formed by the blocks in question. Lakshmibai--Weyman \cite[Theorem 5.3]{lak} give the following expression for the singular locus of $Z_{\a}$:
$$
\operatorname{Sing}(Z_{\a})=\bigcup_{i=1}^{t-1} Z_{\a^i}.
$$
We now discuss the appearances of $\mathcal{L}(\a^i)$ as composition factors of $\mathscr{H}^{\bullet}_{Z_{\a}}(\mc{O}_X)$. By Theorem \ref{mainthm}, the multiplicity of $\mathcal{L}(\a^i)$ in  $\mathscr{H}^{\bullet}_{Z_{\a}}(\mc{O}_X)$ is determined by the fillings of the path $P_i$ with augmented Dyck paths. More precisely, if we write $c=\codim(Z_{\a},X)$, then the multiplicity of $\mathcal{L}(\a^i)$ in $\mathscr{H}^{c+j}_{Z_{\a}}(\mc{O}_X)$ is one if $P_i$ can be filled with an augmented Dyck path with  $j$ bullets, and zero otherwise. A straightforward calculation gives the following generating function for the desired multiplicities:
\begin{equation}
\sum_{j\geq 0} \big[\mathscr{H}^{j}_{Z_{\a}}(\mc{O}_X):\mathcal{L}(\a^i)\big]\cdot q^j=q^c\cdot (q^{|\tilde{a}_i-\tilde{a}_{i+1}-d_{i+1}|}+q^{|\tilde{a}_i-\tilde{a}_{i+1}-d_{i+1}|+2}+\cdots+q^{\tilde{a}_i-\tilde{a}_{i+1}+d_{i+1}-2}),
\end{equation}
where $|\tilde{a}_i-\tilde{a}_{i+1}-d_{i+1}|$ denotes absolute value of the integer $\tilde{a}_i-\tilde{a}_{i+1}-d_{i+1}$. In particular, all cohomological degrees in which $\mathcal{L}(\a^i)$ appears have the same parity.

For instance, in Example \ref{ex1}, we have $\a=(5^1,4^1,2^2)$, $\a^1=(3^2,2^2)$, $\a^2=(5^1,1^3)$, and $c=7$. The multiplicity of $\mathcal{L}(\a^1)$ in local cohomology is described by the generating function $q^7(q^0)$ and the multiplicity of $\mathcal{L}(\a^2)$ in local cohomology is described by the generating function $q^7(q^0+q^2)$.\hfill\mbox{$\diamond$}
\end{example}

In the next example, we discuss the (rationally) smooth case.

\begin{example}
A Schubert variety $\Sv{\a}$ is smooth 	if and only if $\a$ is a rectangle \cite[Theorem 5.3]{lak}, in which case $Z_{\a}$ is isomorphic to a smaller Grassmannian. Observe that if $\a$ is a rectangle with $c=\operatorname{codim}(Z_{\a},X)$, then $\mc{A}(\a;c)=\{(\emptyset)\}$ and $\mc{A}(\a,c+j)=\emptyset$ for all $j\geq 1$. In particular, we recover the following well-known fact in the smooth case:
\begin{equation}\label{ratsmooth}
\mathscr{H}^c_{Z_{\a}}(\mc{O}_X)=\mc{L}(\a),\quad\tn{and}\quad \mathscr{H}^{c+j}_{Z_{\a}}(\mc{O}_X)=0\quad \tn{for $j\geq 1$}.
\end{equation}
The cohomological condition (\ref{ratsmooth}) is referred to as \textit{rational smoothness}. If $Z_{\a}$ is rationally smooth, then by Theorem \ref{mainthm} it follows that $\a$ does not have any inside corners, i.e. $\a$ is a rectangle. In particular, we recover the well-known fact that a Schubert variety in the Grassmannian is rationally smooth if and only if it is smooth (see \cite[Sections 4.3, 5.3]{diagrams} for a survey of rational smoothness properties of Schubert varieties in Hermitian symmetric spaces). Using notation from Section \ref{sec:over}, we have that $Z_w\subseteq G/P$ is rationally smooth if and only if $\mc{GC}_w^{\bullet}$ is a resolution of $\mc{L}(Z_w,G/P)$. \hfill\mbox{$\diamond$}
\end{example}

We prove Theorem \ref{mainthm} in Section \ref{sec:Grass} and we use it to calculate composition factors and weight filtration on local cohomology with support in generic determinantal varieties in Section \ref{sec:det}.

\section{The Hodge-theoretic Grothendieck-Cousin complex}\label{sec:2}

In this section we review background on $\D$-modules, mixed Hodge modules, and local cohomology (Section \ref{sec:D}), and we construct the Hodge-theoretic Grothendieck--Cousin complex (Section \ref{sec:const}). 

For use in our arguments below, our convention is that complex varieties are possibly reducible.

\subsection{$\D$-modules, mixed Hodge modules, and local cohomology}\label{sec:D} Let $X$ be a smooth complex variety of dimension $d_{X}$, with sheaf of algebraic differential operators $\D_{X}$. A $\D_X$-module is a sheaf of left modules over $\D_X$. We write $\tn{MHM}(X)$ for the category of algebraic mixed Hodge modules on $X$ (see \cite[Section 8.3.3]{htt} or \cite{saito89,saito90}), with bounded derived category $\tn{D}^b\tn{MHM}(X)$. These categories are also defined on singular varieties.

Given an irreducible closed subvariety $Z\subseteq X$, we write $\mc{L}(Z,X)$ for the intersection cohomology $\D_X$-module associated to the trivial local system on the regular locus $Z_{\tn{reg}}\subseteq Z$ \cite[Definition 3.4.1]{htt}, and we write $\IC^H_Z$ for the pure Hodge module associated to the trivial variation of Hodge structure on $Z_{\tn{reg}}$ \cite[Section 8.3.3(m13)]{htt}, which has weight $d_Z$. For instance, we have $\mc{L}(X,X)=\mc{O}_X$ and $IC^H_X=\mc{O}_X^H$. For $k\in \mathbb{Z}$ we write $IC^H_Z(k)$ for the $k$-th  Tate twist of $\IC^H_Z$ \cite[Section 8.3.3(m5)]{htt}, which is pure of weight $d_Z-2k$. The modules $\IC^H_Z(k)$ provide a complete list of polarizable pure Hodge modules that may overlie the $\D_X$-module $\mc{L}(Z,X)$ \cite[Section 8.3.3(m13)]{htt}.  

 Let $f:X\to Y$ be a morphism between complex varieties. We write the following for the corresponding direct and inverse image functors for mixed Hodge modules \cite[Section 4]{saito90}
$$
f_{\ast}:\tn{D}^b\tn{MHM}(X)\to \tn{D}^b\tn{MHM}(Y),\;\;\tn{and}\;\; f^{!}:\tn{D}^b\tn{MHM}(Y)\to \tn{D}^b\tn{MHM}(X).
$$
If $X$ and $Y$ are smooth, then these functors lift the $\D$-module functors $\int_f$ and $f^{\dagger}$ from \cite[Chapter 1.5]{htt}, and if $f$ is a closed immersion of codimension $c$, then $f^{!}\mc{O}_Y^H=\mc{O}_X^H(-c)[-c]$.

Let $X$ be smooth and let $i:Z\hookrightarrow X$ be a locally closed subvariety of $X$. Given $\mc{M}\in \tn{D}^b\tn{MHM}(X)$, we define the local cohomology modules of $\mc{M}$ with support in $Z$ via
$$
\lc{Z}^q(\mc{M}):=\mc{H}^q(i_{\ast}i^{!}\mc{M}),\quad q\in \mathbb{Z}.
$$
In particular, for $\mc{M}=\mc{O}^H_X$ this endows each local cohomology module $\lc{Z}^q(\mc{O}_X)$ with a mixed Hodge module structure, which we denote by $\lc{Z}^q(\mc{O}^H_X)$. For more information, see \cite[Section B]{MP4}.

Now assume that $Z\subseteq X$ is a closed subvariety of $X$, with complement $U=X\setminus Z$, and open immersion $j:U\to X$.
Given $\mc{M}\in \tn{D}^b\tn{MHM}(X)$, there is an exact triangle \cite[(4.4.1)]{saito90}
\begin{equation}\label{triangle}
i_{\ast}i^{!}(\mc{M})\lra \mc{M}\lra j_{\ast}j^{!}(\mc{M})\overset{+1}\lra.
\end{equation}
In particular, if $\mc{M}$ is a mixed Hodge module, we obtain an exact sequence 
$$
0\lra \lc{Z}^0(\mc{M})\lra \mc{M}\lra \mc{H}^0(j_{\ast}(\mc{M}|_U))\lra \lc{Z}^1(\mc{M})\lra 0,
$$
and isomorphisms
$$
\lc{Z}^q(\mc{M})\cong \mc{H}^{q-1}(j_{\ast}(\mc{M}|_U)),\quad \tn{for $q\geq 2$.}
$$
We write $F_{\bullet}$ and $W_{\bullet}$ for the Hodge and weight filtrations of $\lc{Z}^q(\mc{M})$.

\subsection{Construction of the Grothendieck--Cousin complex}\label{sec:const}

\begin{lemma}\label{locallyClosedSequence}
Let $\mathcal{M}$ be a mixed Hodge module on a smooth complex variety $X$, and let $Z'\subsetneq Z$ be closed subvarieties of $X$. We have the following exact sequence of mixed Hodge modules:
$$
0 \to \mathscr{H}^0_{Z'}(\mc{M})\to \mathscr{H}^0_Z(\mc{M}) \to \mathscr{H}^0_{Z\setminus Z'}(\mc{M})\to \mathscr{H}^{1}_{Z'}(\mc{M})\to \mathscr{H}^{1}_Z(\mc{M}) \to \mathscr{H}^{1}_{Z\setminus Z'}(\mc{M})\to \cdots 		
$$
\end{lemma}

\begin{proof}
We set $U=X\setminus Z'$ with open immersion $j$, and we write $i':Z'\to X$, $i:Z\to X$ for the natural inclusions. By (\ref{triangle}) applied to $i'=i$ and $i_{\ast}i^!\mc{M}$ we obtain an exact triangle
$$
i'_{\ast}i'^!i_{\ast}i^!\mc{M}\lra i_{\ast}i^!\mc{M}\lra j_{\ast}j^!i_{\ast}i^!\mc{M}\overset{+1}\lra.
$$
Let $f$ be the inclusion of $Z'$ into $Z$. Since $i'=i\circ f$, we have  $i'_{\ast}i'^!i_{\ast}i^!= i'_{\ast}f^!i^!i_{\ast}i^!$, which by adjunction of $i^!$ and $i_{\ast}$ is equal to  $i_{\ast}'i'^!$. Let $\tilde{j}$ be the inclusion of $Z\setminus Z'$ into $Z$, and let $\tilde{i}$ be the inclusion of $Z\setminus Z'$ into $U$. By base change \cite[(4.4.3)]{saito90} we have that $j_{\ast}j^!i_{\ast}i^!=k_{\ast}k^!$. Therefore, the triangle above becomes
\begin{equation}\label{trianglewanted}
i'_{\ast}i'^!\mc{M}\lra i_{\ast}i^!\mc{M}\lra k_{\ast}k^!\mc{M}\overset{+1}\lra.	
\end{equation}
The desired exact sequence is obtained as the long exact sequence of cohomology of (\ref{trianglewanted}). 
\end{proof}

\begin{prop}\label{GCcomplex}
Let $X$ be a smooth complex variety and let 
$$
X=Z_0\supsetneq Z_1 \supsetneq \cdots \supsetneq Z_{\ell} \supsetneq Z_{\ell+1}=\emptyset
$$	
be a decreasing sequence of closed subsets. If $\mathcal{M}$ is a mixed Hodge module on $X$ satisfying $\mathscr{H}^{j-1}_{Z_j}(\mc{M})=\mathscr{H}^{j+1}_{Z_j}(\mc{M})=0$ for all $0\leq j\leq \ell$, then there exists a complex
$$
\mathcal{GC}^{\bullet}_{\{Z\}}(\mathcal{M}):\quad 0  \to \mathscr{H}^0_{Z_0\setminus Z_1}(\mathcal{M})\to \mathscr{H}^1_{Z_1\setminus Z_2}(\mathcal{M}) \to \cdots \to \mathscr{H}^{\ell}_{Z_{\ell}\setminus Z_{\ell+1}}(\mathcal{M})\to 0,
$$
which is a right resolution of $\mathcal{M}$ in the category of mixed Hodge modules. This complex is known as the Grothendieck--Cousin of $\mathcal{M}$ complex associated to the filtration $\{Z\}$.
\end{prop}

\begin{proof}
By Lemma \ref{locallyClosedSequence} and the hypothesis on $\mathcal{M}$, we obtain exact sequences
$$
0 \to \mathscr{H}^j_{Z_j}(\mathcal{M}) \overset{\alpha^j}\lra \mathscr{H}^j_{Z_j\setminus Z_{j+1}}(\mathcal{M})\overset{\beta^j}\lra \mathscr{H}^{j+1}_{Z_{j+1}}(\mathcal{M})\to  0.
$$

For $0\leq j\leq \ell$ we set $\mathcal{GC}^{j}_{\{Z\}}(\mathcal{M})=\mathscr{H}^j_{Z_j\setminus Z_{j+1}}(\mathcal{M})$, and we define $d^j:\mathcal{GC}^{j}_{\{Z\}}(\mathcal{M})\to \mathcal{GC}^{j+1}_{\{Z\}}(\mathcal{M})$ to be the composition of maps
$$
d^j=\alpha^{j+1}\circ \beta^j.
$$
By construction, the image of $d^j$ is  $\mathscr{H}^{j+1}_{Z_{j+1}}(\mathcal{M})$, which is the kernel of $d^{j+1}$. Thus, $\mathcal{GC}^{\bullet}_{\{Z\}}(\mathcal{M})$ is exact in cohomological degrees $j\geq 1$. Since the kernel of $\beta^0$ is $\mathscr{H}^0_{X}(\mathcal{M})=\mathcal{M}$, we conclude that $\mathcal{GC}^{\bullet}_{\{Z\}}(\mathcal{M})$ is a resolution of $\mathcal{M}$.
\end{proof}

The Grothendieck--Cousin complex was introduced and studied in greater generality in \cite{kempf}. In particular, Kempf considers more general input sheaves. We will be interested in the following case.
 
 \begin{corollary}\label{corGCcomplex}
Let $X$ be an irreducible smooth complex variety of dimension $d$, and let 
$$
X=Z_0\supsetneq Z_1 \supsetneq \cdots \supsetneq Z_{d} \supsetneq Z_{d+1}=\emptyset
$$	
be a decreasing sequence of equidimensional closed subsets with $\operatorname{codim}(Z_j,X)=j$. If $Z_j\setminus Z_{j+1}$ is smooth and the inclusion $Z_j\setminus Z_{j+1}\subseteq X$ is affine for all $0\leq j \leq d$, then $\mathcal{GC}^{\bullet}_{\{Z\}}:=\mathcal{GC}^{\bullet}_{\{Z\}}(\mc{O}_X^H)$ is a right resolution of $\mc{O}_X^H$ in the category of mixed Hodge modules.
 \end{corollary}
 
 \begin{proof}
 By Proposition \ref{GCcomplex}, it suffices to show that $\mathscr{H}^{j-1}_{Z_j}(\mc{O}_X^H)=\mathscr{H}^{j+1}_{Z_j}(\mc{O}_X^H)=0$ for all $j$. Since $Z_j$ is equidimensional with $\operatorname{codim}(Z_j,X)=j$ we have that $\mathscr{H}^{j-1}_{Z_j}(\mc{O}_X^H)=0$. 
 
 Let $k$ be the inclusion of $Z_j\setminus Z_{j+1}$ into $X$. Since $Z_j\setminus Z_{j+1}$ is smooth, we have $k^{!}\mc{O}_X^H=\mc{O}_{Z_j\setminus Z_{j+1}}^H(-j)[-j]$, and since $k$ is affine, we have that $k_{\ast}k^{!}\mc{O}_X^H$ only has cohomology in degree $j$, which is $\mathscr{H}^j_{Z_j\setminus Z_{j+1}}(\mc{O}_X^H)$. By Lemma \ref{locallyClosedSequence} we obtain that $\mathscr{H}^i_{Z_j}(\mc{O}_X^H)\cong \mathscr{H}^i_{Z_{j+1}}(\mc{O}_X^H)$ for all $i>j+1$. Since $\mathscr{H}^i_X(\mc{O}_X^H)=0$ for all $i>0$, we conclude that, for all $0\leq j\leq d$ we have $\mathscr{H}^i_{Z_j}(\mc{O}_X^H)=0$ for $i>j$. In particular, $\mathscr{H}^{j+1}_{Z_j}(\mc{O}_X^H)=0$.
 \end{proof}

\section{The Grothendieck--Cousin complex of a Schubert variety}\label{sec:3}

\subsection{Lie algebras and Weyl groups} We let $\g$ be a complex simple Lie algebra of rank $n$. We fix a Cartan subalgebra $\mf{h}\subset \g$, and we write $\Phi\subset \mf{h}^{\ast}$ for the root system of $\g$ with respect to $\mf{h}$. We have the root space decomposition of $\g$:
$$
\g=\mf{h}\oplus \bigoplus_{\alpha \in \Phi} \g_{\alpha}.
$$
We denote by $\Delta=\{\alpha_1,\cdots,\alpha_n\}\subset \Phi$ a choice of simple roots, and we write $\Phi^{+}\subset \Phi$ for the set of positive roots with respect to $\Delta$. The standard Borel subalgebra of $\g$ associated to this data is:
$$
\mf{b}=\mf{h}\oplus \bigoplus_{\alpha \in \Phi^+} \g_{\alpha}.
$$

A subset $I\subset \Delta$ determines a root system $\Phi_I$ via
$$
\Phi_I:=\Phi\cap \sum_{\alpha\in I}\mathbb{Z}\alpha,
$$
and thus determines the following three subalgebras of $\g$:
\begin{equation}
\mf{l}_I=\mf{h}\oplus \bigoplus_{\alpha \in \Phi_I} \g_{\alpha},\quad\quad \mf{u}_I=	\bigoplus_{\alpha \in \Phi^+\setminus \Phi_I}\g_{\alpha},\quad\quad \mf{p}_I=\mf{l}_I\oplus \mf{u}_I.
\end{equation}
We say $\mf{p}_I$ is a parabolic subalgebra of $\g$, with Levi factor $\mf{l}_I$ and nilradical $\mf{u}_I$. When the set $I$ is understood, we will sometimes omit the subscript $I$ and simply write $\mathfrak{p}=\mf{l}\oplus \mf{u}$.

Let $W$ denote the Weyl group of $\g$, generated by the simple reflections $s_{\alpha}$, where $\alpha\in \Delta$. Given $w\in W$ we write $\ell(w)$ for the length of $w$, which is the length of the shortest word for $w$ in the simple reflections. The group $W$ has a unique longest element, which we will denote by $w_{\circ}$, with length $\ell(w_{\circ})=|\Phi^+|$. Given a parabolic subalgebra $\mf{p}_I=\mf{l}\oplus \mf{u}$, we write $W_I=\langle s_{\alpha}\mid \alpha \in I\rangle$ for the Weyl group of the Levi factor $\mf{l}$, with longest element $w_I$. In each left coset $wW_I\in W/W_I$ there exists a unique element of minimal length. We consider the set of minimal length coset representatives:
\begin{equation}
W^I=\{ w\in W\mid \ell(ww)=\ell(w)+\ell(w')\quad \tn{for all $w'\in W_I$}\}.
\end{equation}
Each $x\in W$ may be expressed uniquely as $x=ww'$ where $w\in W^I$ and $w'\in W_I$. We write $w_{\circ}(I)$ for the longest element in $W^I$, which satisfies $w_{\circ}(I)w_I=w_{\circ}$ and $\ell(w_{\circ}(I))=|\Phi^+\setminus \Phi_I|$.

\subsection{Generalized flag varieties and Schubert varieties} Let $G$ be the simply-connected complex simple linear algebraic group with Lie algebra  $\g$, and with Borel subgroup $B$ corresponding to $\mf{b}\subset \g$. For $I\subset \Phi$ we write $P=P_I$ for the parabolic subgroup of $G$ corresponding to $\mf{p}=\mf{p}_I$. We refer to $G/B$ as the complete flag variety of $G$, and we call $G/P$ the generalized flag variety associated to $P$. The dimension of $G/P$ is $\dim(G/P)=\ell(w_{\circ}(I))$.

Let $T$ be the maximal torus of $G$ associated to the fixed Cartan subalgebra $\mf{h}$ of $\g$. With this notation, the Weyl group of $G$ is $W=N_G(T)/T$. The group $G$ has the Bruhat decomposition
$$
G=\bigcup_{w\in W} BwB,
$$
where the $w$ in $BwB$ is any representative in $G$ of the class of $w$. Furthermore, $G/P$ has Bruhat decomposition
$$
G/P=\bigcup_{w\in W^I}BwP/P.
$$
The sets $BwP/P$ are the $B$-orbits on $G/P$, known as the Schubert cells. For $w\in W^I$ we set
\begin{equation}
\Sc{w}=BwP/P\subseteq G/P.
\end{equation}
The Zariski closure $\Sv{w}$ of $\Sc{w}$ is the Schubert variety associated to $w\in W^I$ in $G/P$. It has dimension $\ell(w)$ and codimension $c(w):=\operatorname{codim}(\Sv{w},X)$. We say $v\leq w$ if $O_v\subseteq Z_w$.

\subsection{Structure of the Grothendieck--Cousin complex}
We let $X=G/P_I$ and $d=\dim X$. For $0\leq j\leq d$ we consider the closed sets
$$
Z_j=\bigcup_{\substack{v \in W^I\\ c(v)=j}} \Sv{v}.
$$
The sets $Z_j$ fit into a descending chain
$$
\GG=Z_0\supset Z_1 \supset \cdots \supset Z_d \supset Z_{d+1}=\emptyset,
$$
where the inclusion $Z_j\setminus Z_{j+1}\subseteq X$ is affine. Each $Z_j\setminus Z_{j+1}$ is the disjoint union of closed smooth subvarieties $O_v$ over $v\in W^I$ with $c(v)=j$, so $Z_j\setminus Z_{j+1}$ is smooth and $Z_j$ satisfies $\operatorname{codim}(Z_j,X)=j$.

By Corollary \ref{corGCcomplex}, the mixed Hodge module $\mc{O}^H_X$ admits a right resolution 
$$
\mathcal{GC}^{\bullet}_{\{Z\}}:\quad 0  \lra \mathscr{H}^0_{Z_0\setminus Z_1}(\mc{O}_X^H)\lra  \mathscr{H}^1_{Z_1\setminus Z_2}(\mc{O}_X^H) \lra  \cdots \lra  \mathscr{H}^d_{Z_d\setminus Z_{d+1}}(\mc{O}_X^H)\lra  0,
$$
where we view $\mathscr{H}^j_{Z_j\setminus Z_{j+1}}(\mc{O}_X^H)$ as being in cohomological degree $j$. For $w\in W^I$ we define
\begin{equation}
\mathcal{GC}^{\bullet}_{w}=\mathscr{H}^0_{\Sv{w}}\big(\mathcal{GC}^{\bullet}_{\{Z\}}\big).
\end{equation}
The main result of this subsection is the following.
\begin{theorem}\label{GCaComplex}
The following is true about $\mathcal{GC}^{\bullet}_{w}$.
\begin{enumerate}
\item For all $0\leq j\leq d$ we have
$$
\mc{H}^j\big(\mathcal{GC}^{\bullet}_{w}\big)=\mathscr{H}^j_{\Sv{w}}(\mc{O}_X^H).
$$
\item $\mathcal{GC}^j_{w}=0$ for $j<c(w)$, and for $j\geq c(w)$ we have an isomorphism
$$
\mathcal{GC}^j_{w}=\bigoplus_{\substack{v \leq w \\ c(v)=j}} \mathscr{H}^j_{\Sc{v}}(\mc{O}_X^H).
$$
\item For $u<v \leq w$ with $j=c(v)=c(u)-1$, the differential in $\mathcal{GC}^{\bullet}_{w}$ induces a nonzero morphism:
$$
\mathscr{H}^j_{\Sc{v}}(\mc{O}_X^H)\lra \mathscr{H}^{j+1}_{\Sc{u}}(\mc{O}_X^H),
$$
with $\mathcal{L}(Z_u,X)$ in its image.
\end{enumerate}
	
\end{theorem}

We note that part (3) is not formal: its proof uses the fact that Schubert varieties are normal. In order to prove Theorem \ref{GCaComplex}, we first need the following lemma.

\begin{lemma}\label{Lem:codimSplit}
For $0\leq j\leq d$ we have
$$
\mathscr{H}^j_{Z_j\setminus Z_{j+1}}(\mc{O}_X^H)=\bigoplus_{\substack{v \in W^I\\ c(v)=j}} \mathscr{H}^j_{\Sc{v}}(\mc{O}_X^H).
$$	
\end{lemma}

\begin{proof}
Let $k^j:Z_j\setminus Z_{j+1}\hookrightarrow X$, $i_v:O_v\hookrightarrow X$, and $i^j_v: O_v\hookrightarrow Z_j\setminus Z_{j+1}$. The variety $Z_j\setminus Z_{j+1}$ is the disjoint union of closed smooth subvarieties $O_v$ over $v\in W^I$ with $c(v)=j$. Thus,
$$
\mc{O}^H_{Z_j\setminus Z_{j+1}}=\bigoplus_{\substack{v \in W^I\\ c(v)=j}} i^j_{v\ast}\mc{O}^H_{O_v}.
$$
Since $\mathscr{H}^j_{Z_j\setminus Z_{j+1}}(\mc{O}_X^H)=k^j_{\ast} \mc{O}^H_{Z_j\setminus Z_{j+1}}(-j)[-j]$ and $\mathscr{H}^j_{\Sc{v}}(\mc{O}_X^H)=i_{v\ast}\mc{O}^H_{O_v}(-j)[-j]$, we obtain the desired result using $k^j_{\ast}\circ i^j_{v\ast}=i_{v\ast}$.
\end{proof}

\begin{proof}[Proof of Theorem \ref{GCaComplex}]
Since $\mathcal{GC}^{\bullet}_{\{Z\}}$ is a resolution of $\mc{O}_X^H$, to prove (1) and (2) it suffices to show that
$$
\mathscr{H}^0_{\Sv{w}}( \mathscr{H}^j_{Z_j\setminus Z_{j+1}}(\mc{O}_X^H)\big)=\bigoplus_{\substack{v \leq w \\ c(v)=j}}\mathscr{H}^j_{\Sc{v}}(\mc{O}_X^H),
$$	
and $\mathscr{H}^i_{\Sv{w}}(\mathscr{H}^{j}_{Z_j\setminus Z_{j+1}}(\mc{O}_X^H))=0$ for $i>0$. By Lemma \ref{Lem:codimSplit} we have that
$$
\mathscr{H}^j_{Z_j\setminus Z_{j+1}}(\mc{O}_X^H)=\bigoplus_{\substack{v \in W^I\\ c(v)=j}} \mathscr{H}^j_{\Sc{v}}(\mc{O}_X^H).
$$
Let $v\in W^I$ with $c(v)=j$. If $v\leq w$, then $\Sc{v}\cap \Sv{w}=\Sc{v}$. Let $\beta:\Sc{v}\hookrightarrow \Sv{w}$, $\alpha:\Sc{v}\hookrightarrow X$, $\gamma:\Sv{w}\hookrightarrow X$, so that $\alpha=\gamma\circ \beta$. By base change \cite[(4.4.3)]{saito90} we have $\gamma_{\ast}\gamma^!\alpha_{\ast}\alpha^!=\gamma_{\ast}\beta_{\ast}\alpha^!=\alpha_{\ast}\alpha^!$. Thus, there is a Grothendieck spectral sequence
$$
E^{p,q}_2=\mathscr{H}^p_{\Sv{w}}\big(\mathscr{H}^q_{\Sc{v}}(\mc{O}_X^H)\big) \implies \mathscr{H}^{p+q}_{\Sc{v}\cap \Sv{w}}(\mc{O}_X^H)=\mathscr{H}^{p+q}_{\Sc{v}}(\mc{O}_X^H).
$$
Since $\mathscr{H}^q_{\Sc{v}}(\mc{O}_X^H)\neq 0$ only if $q=j$, this spectral sequence has only one nonzero column on the $E_2$-page, and has only one nonzero row on the $E_{\infty}$-page. Thus, the spectral sequence degenerates at $E_2$, so we have $\mathscr{H}^0_{\Sv{w}}(\mathscr{H}^j_{\Sc{v}}(\mc{O}_X^H))=\mathscr{H}^{j}_{\Sc{v}}(\mc{O}_X^H)$, and $\mathscr{H}^i_{\Sv{w}}(\mathscr{H}^{j}_{\Sc{v}}(\mc{O}_X^H))=0$ for $i>0$. On the other hand, if $v\nleq w$, then $O_v\cap Z_w=\emptyset$, so by base change we obtain a spectral sequence 
$$
E^{p,q}_2=\mathscr{H}^p_{\Sv{w}}\big(\mathscr{H}^q_{\Sc{v}}(\mc{O}_X^H)\big) \implies 0.
$$
Since $\mathscr{H}^q_{\Sc{v}}(\mc{O}_X^H)\neq 0$ only if $q=j$, this spectral sequence is degenerate, yielding $\mathscr{H}^i_{\Sv{w}}(\mathscr{H}^{j}_{\Sc{v}}(\mc{O}_X^H))=0$ for all $i,j$, which completes the proof of (1) and (2).

It suffices to prove (3) in the case $w=w_{\circ}(I)$ is the longest element of $W^I$ so that $\mathcal{GC}^{\bullet}_{w}=\mathcal{GC}^{\bullet}_{\{Z\}}$, and after restricting to the following open set $$
U=\GG \setminus  \bigcup_{x \not\geq u} \Sv{x}.
$$
We note that $O_u$ is closed in $U$. By construction of $\mathcal{GC}^{\bullet}_{\{Z\}}$, the map from $\mathscr{H}^j_{\Sc{v}}(\mc{O}_U^H)$ to $\mathscr{H}^{j+1}_{\Sc{u}}(\mc{O}_U^H)$ in $\mathcal{GC}^{\bullet}_{\{Z\}}|_U$ fits into the following exact sequence (see Lemma \ref{locallyClosedSequence}):
 $$
 0\lra \mathscr{H}^j_{\overline{O}_v}(\mc{O}_U^H)\lra  \mathscr{H}^j_{\Sc{v}}(\mc{O}_U^H)\lra \mathscr{H}^{j+1}_{\Sc{u}}(\mc{O}_U^H)\lra \mathscr{H}^{j+1}_{\overline{O}_v}(\mc{O}_U^H)\lra 0.
 $$
Since $\Sv{v}$ is normal \cite[Theorem 3]{normal}, its singular locus must have dimension $\leq \ell(v)-2$, and by equivariance it must be a union of Schubert varieties. Since $\ell(v)=\ell(u)+1$, it follows that  $\overline{O}_v$ is smooth in $U$, so we have $\mathscr{H}^{j+1}_{\overline{O}_v}(\mc{O}_U^H)=0$. As $\mathscr{H}^{j+1}_{\Sc{u}}(\mc{O}_U^H)\cong \mathcal{L}(Z_u,X)|_U\neq 0$, we conclude that the map $\mathscr{H}^j_{\Sc{v}}(\mc{O}_U^H)\to \mathscr{H}^{j+1}_{\Sc{u}}(\mc{O}_U^H)$ is surjective, and therefore the desired map is nonzero, with $\mathcal{L}(Z_u,X)$ in its image.
\end{proof}

Since $\mc{GC}_w^{\bullet}$ is a complex of mixed Hodge modules, it is endowed with a Hodge filtration $F_{\bullet}$ and a weight filtration $W_{\bullet}$.
The next result gives our main tool to get information about the mixed Hodge module structure on local cohomology. The following result implies Theorem A.

\begin{theorem} \label{prop:degen}
Let $w\in W^I$, let $p\in \mathbb{Z}$, and let $q\geq 0$. The following is true about the Hodge and weight filtrations on local cohomology with support in $\Sv{w}$:
$$
F_p(\lc{Z_w}^{q}(\mc{O}_X^H))\cong \mc{H}^q(F_p(\mathcal{GC}_w^{\bullet})),\quad\tn{and}\quad W_p(\lc{Z_w}^{q}(\mc{O}_X^H))\cong \mc{H}^q(W_p(\mathcal{GC}_w^{\bullet})).
$$
Furthermore, the following is true for $i\in \mathbb{Z}$:
$$
F_pW_i(\lc{Z_w}^{q}(\mc{O}_X^H))\cong \mc{H}^q(F_pW_i(\mathcal{GC}_w^{\bullet})),
$$
where the left side is the intersection of the two filtrations.
\end{theorem}
\begin{proof}
This is an immediate consequence of the fact that morphisms in the category of mixed Hodge modules are bi-strict with respect to the Hodge and weight filtrations \cite[Section 8.3.3(m4)]{htt}.
\end{proof}


\subsection{Parabolic Verma modules, $P$-orbits, and localization}\label{sec:loc} Let $U(\mf{g})$ be the universal enveloping algebra of $\mf{g}$, and let $\mf{p}=\mf{p}_I$ be a parabolic subalgebra of $\mf{g}$ with Levi decomposition $\mf{p}=\mf{l}\oplus \mf{u}$ and adjoint group $P=P_I$. Let $\tn{mod}_P(\mathfrak{g},\chi_{0})$ denote the category of $P$-equivariant finitely-generated $U(\mathfrak{g})$-modules with trivial central character (see \cite[Section 11.5]{htt}). We write $\rho=\frac{1}{2}\sum_{\alpha\in \Phi^+}\alpha$.

For $w\in W^I$ we let $F_{w}$ denote the finite dimensional irreducible $\mf{l}$-representation with highest weight $ww_I\rho-\rho$, where $w_I$ is the longest element in $W_I$. We view $F_{w}$ as a representation of $\mf{p}$, with $\mf{u}$ acting trivially. The parabolic Verma module (sometimes called generalized Verma module) with highest weight $ww_I\rho-\rho$ is the left $U(\mathfrak{g})$-module  
\begin{equation}
M_I(w)=U(\g)\otimes _{U(\mf{p})}F_{w}.
\end{equation}
The module $M_I(w)$ is a quotient of the Verma module $M_{\emptyset}(w)$, and thus has unique simple quotient $L(w)$, the irreducible representation of $\g$ with highest weight $ww_I\rho-\rho$. We define the BGG dual of $M_I(w)$ as $N_I(w)=M_I(w)^{\vee}$ (see \cite[Section 3.2]{humphreys}). The module $N_I(w)$ is dual to $M_I(w)$ in the sense that they have the same composition factors, multiplicities of factors, and the submodule lattice of $N_I(w)$ is dual to that of $M_I(w)$ (see \cite[Section 3.2 Theorem]{humphreys}).

We now state the Beilinson--Bernstein theorem \cite{BB} for trivial central characters. Consider the natural left action of $P$ on the complete flag variety $G/B$, with orbits
\begin{equation}\label{porbits}
G/B=\bigcup_{w\in W^I}\mathbb{O}_w,\quad\tn{where}\quad \mathbb{O}_w=\bigcup_{z\in W_I} \Sc{wz},
\end{equation}
where $\Sc{y}$ denotes the $B$-orbit of $yB$ in $G/B$. Here, $\Sc{ww_I}$ is dense in $\mathbb{O}_w$, so that $\ol{O}_{ww_I}=\ol{\mathbb{O}}_w=\Sv{w}$ in $G/B$. For $w\in W^I$, we define the $\D_{G/B}$-modules:
\begin{equation}
\mathcal{N}_w=\lc{\mathbb{O}_w}^{c(w)}(\mc{O}_{G/B}),\quad \mc{M}_w=\mathbb{D}\mc{N}_w,\quad \mathcal{L}_w=\mc{L}(\mathbb{O}_w,G/B).
\end{equation}
For the following statements, we write $\tn{mod}_P(\D_{G/B})$ for the category of $P$-equivariant coherent $\D_{G/B}$-modules (see \cite[Section 2.1]{categories} or \cite[Section 11.5]{htt}), and we write $\tn{mod}_B(\D_{G/P})$ for the category of $B$-equivariant coherent $\D_{G/P}$-modules.
\begin{theorem}\label{local}
The following is true about $\tn{mod}_P(\D_{G/B})$.
\begin{enumerate}
\item The global sections functor $\Gamma(G/B,-)$ induces an equivalence of categories
$$
\tn{mod}_P(\D_{G/B})\cong \tn{mod}_P(\mathfrak{g},\chi_0).
$$	
\item Furthermore, given $w\in W^I$ we have
$$
\Gamma(G/B,\mathcal{N}_w)\cong N_I(w),\quad \Gamma(G/B,\mathcal{M}_w)\cong M_I(w),\quad \Gamma(G/B,\mathcal{L}_w)\cong L(w).
$$
\end{enumerate}
\end{theorem}

\begin{proof}
The equivalence of categories in (1) is a special case of the Beilinson--Bernstein theorem \cite{BB} (see \cite[Theorem 11.5.3]{htt} for a textbook account). For the first two isomorphisms in (2), see \cite[(3.27) Lemma]{CC} and subsequent discussion. The final isomorphism follows from \cite[Proposition]{BB} (alternatively, 	\cite[Prop.12.3.2]{htt}, noting that $\mc{L}(\mathbb{O}_w,G/B)=\mc{L}(O_{ww_I},G/B)$). 
\end{proof}

For the remainder of the subsection we write $\pi^{\ast}$ for the (un-shifted) pullback functor of $\D$-modules \cite[Section 1.3]{htt} and we write $\pi_{\ast}$ for the direct image functor of quasi-coherent $\mc{O}$-modules.

\begin{theorem}\label{basechange}
Consider the natural projection $\pi: G/B \to G/P$. Then $\pi^{\ast}$ is exact and  induces an equivalence of categories 
$$
\tn{mod}_B(\D_{G/P})\cong \tn{mod}_P(\D_{G/B}).
$$	
Furthermore, given $w\in W^I$ we have
$$
\pi^{\ast}\lc{\Sc{w}}^{c(w)}(\mc{O}_{G/P})\cong \mc{N}_w,\quad \pi^{\ast}\big(\mathbb{D}\big(\lc{\Sc{w}}^{c(w)}(\mc{O}_{G/P})\big)\big)\cong \mc{M}_w,\quad \pi^{\ast}\mc{L}(\Sc{w},G/P)\cong \mathcal{L}_w.
$$

\end{theorem}

\begin{proof}
Since $\pi$ is smooth, the functor $\pi^{\ast}$ is exact \cite[Proposition 1.5.13]{htt}. By \cite[Proposition 4.5(b)]{categories}, the functor $\pi^{\ast}$ induces the desired equivalence of categories.

It follows from (\ref{porbits}) that $\pi^{-1}(O_w)=\mathbb{O}_w$, so base change \cite[Theorem 1.7.3]{htt} yields the first two isomorphisms. The third isomorphism follows from the first and the equivalence of categories.
\end{proof}

Combining Theorem \ref{local} and Theorem \ref{basechange}, we obtain the following.

\begin{corollary}\label{locP}
The following is true about $\tn{mod}_B(\D_{G/P})$.
\begin{enumerate}
\item The global sections functor $\Gamma(G/P,-)$ induces an equivalence of categories
$$
\tn{mod}_B(\D_{G/P})\cong \tn{mod}_P(\mathfrak{g},\chi_0).
$$	
\item Furthermore, given $w\in W^I$ we have
$$
\Gamma\big(G/P,\lc{\Sc{w}}^{c(w)}(\mc{O}_{G/P})\big)\cong N_I(w),\quad \Gamma\big(G/P,\mathbb{D}\big(\lc{\Sc{w}}^{c(w)}(\mc{O}_{G/P})\big)\big)\cong M_I(w),\quad \Gamma(G/P,\mathcal{L}_w)\cong L(w).
$$
\end{enumerate}

\end{corollary}

\begin{proof}
We follow the strategy described in \cite[footnote pg. 14]{borhoBry} (see also \cite[Proposition 2.14]{holland}). Let $\pi:G/B\to G/P$ be the  projection. For $\mc{M}\in \tn{mod}_B(\D_{G/P})$ we have that $\mathbb{R}^j\pi_{\ast}(\pi^{\ast}\mc{M})=0$ for $j>0$ (see \cite[Proposition 2.14(a)]{holland}). By Theorem \ref{local}(1) we have that the Leray--Serre spectral sequence
$$
H^i(G/P,\mathbb{R}^j\pi_{\ast}(\pi^{\ast}\mc{M}))\implies H^{i+j}(G/B,\pi^{\ast}\mc{M}),
$$
degenerates to yield an isomorphism $\Gamma(G/P,\mc{M})\cong \Gamma(G/B, \pi^{\ast}\mc{M})$ as $U(\mathfrak{g})$-modules. By Theorem \ref{local} and Theorem \ref{basechange} this proves both (1) and (2).
	\end{proof}

For $w\in W^I$ we consider the \defi{global Grothendieck--Cousin complex} $GC_w^{\bullet}=\Gamma(X,\mathcal{GC}^{\bullet}_w)$:
\begin{equation}
GC^{\bullet}_w:\quad 0  \longrightarrow N_I(w) \longrightarrow \bigoplus_{\substack{v\leq w\\ c(v)=c(w)+1}} N_I(v)\longrightarrow \cdots \longrightarrow \bigoplus_{\substack{v\leq w\\ c(w)=\ell(w_I)-1}} N_I(v) \longrightarrow N_I(e)\longrightarrow  0,
\end{equation}
We define the Hodge filtration $F_{\bullet}$ and weight filtration $W_{\bullet}$ on $N_I(w)$ via (see \cite[Theorem 1.1]{davis}):
\begin{equation}
F_p(N_I(w))=\Gamma(F_p(\lc{\Sc{w}}^{c(w)}(\mc{O}_{X}))),\quad \tn{and}\quad W_p(N_I(w))=\Gamma(W_p(\lc{\Sc{w}}^{c(w)}(\mc{O}_{X}))).
\end{equation} 
We endow $M_I(w)$ with filtrations in similar manner. The weight filtration on $M_I(w)$ is called the Jantzen filtration (see \cite{jantzen, humphreys}). Similarly, we define Hodge and weight filtrations on the global local cohomology modules for $q\geq 0$:
\begin{equation}
F_p(H_{Z_w}^{q}(X,\mathcal{O}_X))=\Gamma(F_p(\mathscr{H}^q_{Z_w}(\mc{O}_X))),\quad\tn{and}\quad W_p(H_{Z_w}^{q}(X,\mathcal{O}_X))=\Gamma(W_p(\mathscr{H}^q_{Z_w}(\mc{O}_X))).	
\end{equation}

By Theorem \ref{GCaComplex}, Corollary \ref{locP}, and \cite[Theorem 1.1]{davis} we obtain:

\begin{corollary}\label{globalGC}
\begin{enumerate}
\item For $w\in W^I$ we have that the morphisms in $GC^{\bullet}_w$ are strict with respect to the Hodge and weight filtrations. In particular,
$$
\tn{gr}^F_p(H_{Z_w}^{q}(X,\mathcal{O}_X))\cong \mc{H}^q(\tn{gr}^F_p(GC_w^{\bullet})),\quad \tn{and}\quad\tn{gr}^W_p(H_{Z_w}^{q}(X,\mathcal{O}_X))\cong \mc{H}^q(\tn{gr}^W_p(GC_w^{\bullet})).
$$

\item For $u<v \leq w$ with $j=c(v)=c(u)-1$, the differential in $GC^{\bullet}_{w}$ induces a nonzero morphism 
$$
N_I(v)\to N_I(u).
$$	

\item For $v\leq w$ the multiplicity of $L(v)$ as a composition factor of $\tn{gr}^W_p(H_{Z_w}^{q}(X,\mathcal{O}_X))$ is equal to the multiplicity of $\mc{L}(O_v,X)$ as a composition factor of $\tn{gr}^W_p(\mathscr{H}^q_{Z_w}(\mc{O}_X^H))$.

\end{enumerate}

\end{corollary}

We will not pursue the Hodge filtration further in this work.

\section{Calculations for the Grassmannian} \label{sec:Grass}

Let $X=\tn{Gr}(k,n)$ be the Grassmannian of $k$-dimensional subspaces of an $n$-dimensional complex vector space, so that $d_X=k(n-k)$. We prove Theorem \ref{mainthm} and apply it in Section \ref{sec:det} to calculate local cohomology with support in determinantal varieties. We freely use notation from Section \ref{sec:state}. 

\subsection{The Grassmannian and minimal length coset representatives}\label{sec:young}  Let $\mf{g}=\mf{sl}_n(\C)$  with simple roots $\Delta=\{\alpha_1,\cdots,\alpha_{n-1}\}$, where $\alpha_i=\epsilon_i-\epsilon_{i+1}$. Let $I=\Delta\setminus \{\alpha_k\}$ and let $\mf{p}=\mf{p}_I$. We write $G=\operatorname{SL}_n(\C)$ and $P$ for the groups corresponding to $\mf{g}$ and $\mf{p}$, so that $X\cong G/P$.

The set of minimal length coset representatives of $W/W_I$ is:
\begin{equation}
W^I=\{(i_1,\cdots ,i_n)\in \mathfrak{S}_n\mid i_1<\cdots <i_k,\quad i_{k+1}<\cdots<i_n\},
\end{equation}
where $\mf{S}_n$ is the symmetric group on $n$ elements. In our notation, the identity $e$ is represented by $(1,2,\cdots, n)$, and the longest element $w_{\circ}(I)$ is represented by $(n-k+1,\cdots, n,1,\cdots ,n-k)$. To each $w=(i_1,\cdots,i_n)\in W^I$ we associate a Young diagram $\a=\a(w)$:
$$
a_1=i_k-k,\quad a_2=i_{k-1}-(k-1),\quad \cdots, \quad a_k=i_1-1,
$$
i.e. $a_j=i_{k+1-j}-(k+1-j)$ for $1\leq j\leq k$ and $a_j=0$ for $j>k$. For instance, we have
$$
e \;\longleftrightarrow \;(0,\cdots,0),\quad \tn{and}\quad w_{\circ}(I)\; \longleftrightarrow \; (n-k,\cdots,n-k).
$$
We obtain a bijection between $W^I$ and the set
\begin{equation}
\mathcal{R}_{k,n}=\{\a \mid n-k\geq a_1\geq a_2\geq \cdots \geq a_k\geq 0\}.	
\end{equation}
We think of $\mathcal{R}_{k,n}$ as the set of Young diagrams that fit into a $k\times (n-k)$ rectangle.

Given $w\in W^I$ with $\a=\a(w)$ we write $\Sc{\a}=\Sc{w}$ and $\Sv{\a}=\Sv{w}$ for the corresponding Schubert cell and Schubert variety. We have $\dim \Sv{\a}=|\a|=\ell(w)$. We write $M(\a)=M_I(w)$, $N(\a)=N_I(w)$, $L(\a)=L_I(w)$, so that $M(\a)$ and $L(\a)$ have highest weight
\begin{equation}
\lambda_{\a,n}:=ww_I\rho-\rho=(-a_k^c,-a_{k-1}^c,\cdots, -a_1^c\mid (a^c)_1',(a^c)_2',\cdots, (a^c)_{n-k}'),
\end{equation}
where $\underline{a}^c=(n-k-a_k,\cdots,n-k-a_1)$ is the complementary partition of $\a$ in $\mathcal{R}_{k,n}$ and $(\a^c)'$ is the conjugate partition to $\a^c$. We use the notation $\lambda_{\a,n}$ to emphasize the dependence on $n$.

\subsection{Properties of the parabolic Verma modules}

Let $\a\in \mc{R}_{k,n}$. We freely identify the modules $N(\a)$, $M(\a)$, $L(\a)$ with their $\D_X$-module theoretic counterparts via Corollary \ref{locP}, i.e.
$$
 N(\a)\;\longleftrightarrow \;\mathscr{H}^{c(\a)}_{\Sc{\a}}(\mc{O}_X),\quad M(\a)\;\longleftrightarrow \; \mathbb{D}\big(\mathscr{H}^{c(\a)}_{\Sc{\a}}(\mc{O}_X)\big),\quad L(\a)\;\longleftrightarrow \; \mathcal{L}(\Sc{\a},X),
$$
and we recall that $M(\a)$ and $N(\a)$ are endowed with weight filtrations $W_{\bullet}$ via this equivalence.

A Loewy filtration on a finite length $U(\mf{g})$-module $M$ is a filtration of minimal length such that the subquotients are semi-simple. Two examples of Loewy filtrations are the socle and radical filtrations (see \cite[Section 8.14]{humphreys}). We say that $M$ is rigid if the socle and radical filtrations coincide, in which case there is a unique Loewy filtration on $M$. For $p\in \mathbb{Z}$ we write
\begin{equation}
\mathcal{Z}_p(\a)=\{\mathbb{D}\in Dyck(\a) \mid |\mathbb{D}|=p+|\a|-2d_{X} \}.
\end{equation}
We emphasize that elements of $\mathcal{Z}_p(\a)$ may have Dyck paths of length one, and they have no bullets. Recall that $\a^{\mathbb{D}}$ is the partition obtained from $\a$ by removing the boxes in the support of $\mathbb{D}$ (see (\ref{defDa})).

The following result is a collection of facts from the literature. 

\begin{theorem}\label{keylemma}
The modules $M(\a)$ and $N(\a)$ are rigid, and the weight filtration $W_{\bullet}$ is the unique Loewy filtration. For $p\in \mathbb{Z}$ the graded pieces of the weight filtration are given by
$$
 \tn{gr}_p^W(M(\a))=\bigoplus_{\mathbb{D}\in \mathcal{Z}_{-p}(\a)} L(\a^{\mathbb{D}}),\quad\tn{and}\quad \tn{gr}_p^W(N(\a))=\bigoplus_{\mathbb{D}\in \mathcal{Z}_p(\a)} L(\a^{\mathbb{D}}).
$$
In particular, each composition factor of $M(\a)$ and $N(\a)$ appears with multiplicity one. 
\end{theorem}

\begin{proof}
The composition factors of $M(\a)$ (and hence $N(\a)$) are described in \cite[Theorem 1.3]{cisWeight}. The authors also show that the socle, radical, and \textit{$\ell$-adic weight} filtrations coincide and are the unique Loewy filtration \cite[Corollary 1.7]{cisWeight}.  The graded pieces of the Loewy filtration are interpreted in terms of inverse parabolic Kazhdan--Lusztig polynomials in \cite[Corollary 7.1.3]{irvingBook}. For the Dyck pattern interpretation of these inverse Kazhdan--Lusztig polynomials, see \cite[Corollary 1]{path} (see also \cite[Section 3.1]{path}). In \cite[Theorem 1.4]{KT} it is shown that these inverse Kazhdan--Lusztig polynomials also describe the weight filtration on $N(\a)$ (and hence $M(\a)$). Thus, the weight filtration is the unique Loewy filtration. We have normalized the shift on $W_{\bullet}$ so that $L(\a)$ lives in weight $p=2d_{\GG}-|\a|$ of $N(\a)$. 
\end{proof}

We remark that the original combinatorial description of the relevant Kazhdan--Lusztig polynomials  may be found in \cite{lascoux}.

\begin{example}\label{ex:comp} Let $\a=(3,2,2)$. The following example works equally well for any $k,n$ for which $\a\in \mc{R}_{k,n}$. For concreteness, let $k=3$ and $n=6$, so that $d_X=9$. The elements of $Dyck(\a)$ are:\\

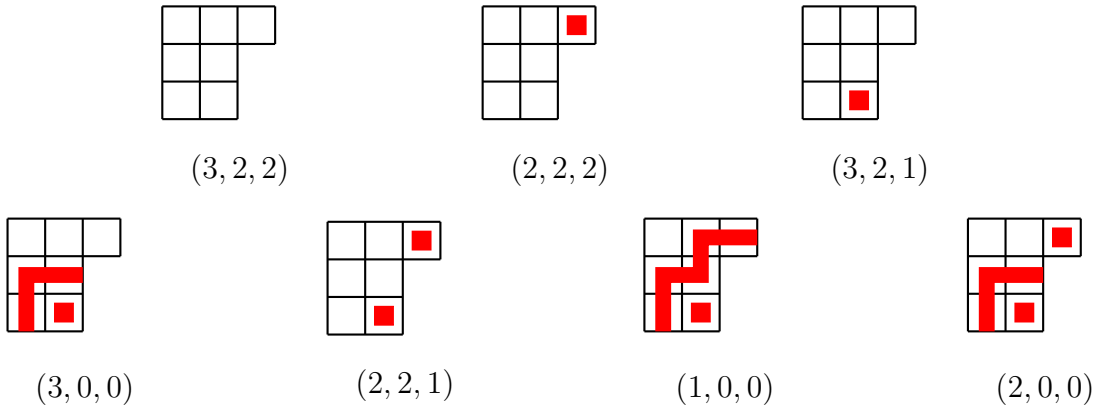
\begin{center}
\begin{minipage}{.18\textwidth}
\centering 
\begin{tikzpicture}[x=\unitsize,y=\unitsize,baseline=0]
\tikzset{vertex/.style={}}%
\tikzset{edge/.style={  thick}}%
\draw[edge] (0,0) -- (4,0);
\draw[edge] (0,2) -- (4,2);
\draw[edge] (0,4) -- (6,4);
\draw[edge] (0,6) -- (6,6);

\draw[edge] (0,0) -- (0,6);
\draw[edge] (2,0) -- (2,6);
\draw[edge] (4,0) -- (4,6);
\draw[edge] (6,4) -- (6,6);
\end{tikzpicture}
\captionsetup{labelformat=empty}	
\captionof{figure}{$(3,2,2)$}	
\end{minipage}
\quad\quad
\begin{minipage}{.18\textwidth}
\centering 
\begin{tikzpicture}[x=\unitsize,y=\unitsize,baseline=0]
\tikzset{vertex/.style={}}%
\tikzset{edge/.style={  thick}}%
\draw[edge] (0,0) -- (4,0);
\draw[edge] (0,2) -- (4,2);
\draw[edge] (0,4) -- (6,4);
\draw[edge] (0,6) -- (6,6);
\draw[color=red,fill=red] (4.5,4.5) rectangle (5.5,5.5);
\draw[edge] (0,0) -- (0,6);
\draw[edge] (2,0) -- (2,6);
\draw[edge] (4,0) -- (4,6);
\draw[edge] (6,4) -- (6,6);
\end{tikzpicture}
\captionsetup{labelformat=empty}	
\captionof{figure}{$(2,2,2)$}
\end{minipage}
\quad\quad
\begin{minipage}{.18\textwidth}
\centering 
\begin{tikzpicture}[x=\unitsize,y=\unitsize,baseline=0]
\tikzset{vertex/.style={}}%
\tikzset{edge/.style={  thick}}%
\draw[edge] (0,0) -- (4,0);
\draw[edge] (0,2) -- (4,2);
\draw[edge] (0,4) -- (6,4);
\draw[edge] (0,6) -- (6,6);
\draw[color=red,fill=red] (2.5,0.5) rectangle (3.5,1.5);
\draw[edge] (0,0) -- (0,6);
\draw[edge] (2,0) -- (2,6);
\draw[edge] (4,0) -- (4,6);
\draw[edge] (6,4) -- (6,6);
\end{tikzpicture}
\captionsetup{labelformat=empty}	
\captionof{figure}{$(3,2,1)$}
\end{minipage}
\medskip
\medskip

\begin{minipage}{.18\textwidth}
\centering 
\begin{tikzpicture}[x=\unitsize,y=\unitsize,baseline=0]
\tikzset{vertex/.style={}}%
\tikzset{edge/.style={  thick}}%
\draw[edge] (0,0) -- (4,0);
\draw[edge] (0,2) -- (4,2);
\draw[edge] (0,4) -- (6,4);
\draw[edge] (0,6) -- (6,6);
\draw[color=red, fill=red] (2.5,0.5) rectangle (3.5,1.5);
\draw[edge] (0,0) -- (0,6);
\draw[edge] (2,0) -- (2,6);
\draw[edge] (4,0) -- (4,6);
\draw[edge] (6,4) -- (6,6);
\draw[red, line width=6pt] (1,0) -- (1,3) -- (4,3) ;
\end{tikzpicture}
\captionsetup{labelformat=empty}	
\captionof{figure}{$(3,0,0)$}	
\end{minipage}
\quad\quad
\begin{minipage}{.18\textwidth}
\centering 
\begin{tikzpicture}[x=\unitsize,y=\unitsize,baseline=0]
\tikzset{vertex/.style={}}%
\tikzset{edge/.style={  thick}}%
\draw[edge] (0,0) -- (4,0);
\draw[edge] (0,2) -- (4,2);
\draw[edge] (0,4) -- (6,4);
\draw[edge] (0,6) -- (6,6);
\draw[color=red,fill=red] (2.5,0.5) rectangle (3.5,1.5);
\draw[color=red,fill=red] (4.5,4.5) rectangle (5.5,5.5);
\draw[edge] (0,0) -- (0,6);
\draw[edge] (2,0) -- (2,6);
\draw[edge] (4,0) -- (4,6);
\draw[edge] (6,4) -- (6,6);
\end{tikzpicture}
\captionsetup{labelformat=empty}	
\captionof{figure}{$(2,2,1)$}
\end{minipage}
\quad\quad
\begin{minipage}{.18\textwidth}
\centering 
\begin{tikzpicture}[x=\unitsize,y=\unitsize,baseline=0]
\tikzset{vertex/.style={}}%
\tikzset{edge/.style={  thick}}%
\draw[edge] (0,0) -- (4,0);
\draw[edge] (0,2) -- (4,2);
\draw[edge] (0,4) -- (6,4);
\draw[edge] (0,6) -- (6,6);
\draw[color=red,fill=red] (2.5,0.5) rectangle (3.5,1.5);
\draw[edge] (0,0) -- (0,6);
\draw[edge] (2,0) -- (2,6);
\draw[edge] (4,0) -- (4,6);
\draw[edge] (6,4) -- (6,6);
\draw[red, line width=6pt] (1,0) -- (1,3) -- (3,3) -- (3,5) -- (6,5);
\end{tikzpicture}
\captionsetup{labelformat=empty}	
\captionof{figure}{$(1,0,0)$}
\end{minipage}
\quad\quad
\begin{minipage}{.18\textwidth}
\centering 
\begin{tikzpicture}[x=\unitsize,y=\unitsize,baseline=0]
\tikzset{vertex/.style={}}%
\tikzset{edge/.style={  thick}}%
\draw[edge] (0,0) -- (4,0);
\draw[edge] (0,2) -- (4,2);
\draw[edge] (0,4) -- (6,4);
\draw[edge] (0,6) -- (6,6);
\draw[color=red,fill=red] (2.5,0.5) rectangle (3.5,1.5);
\draw[color=red,fill=red] (4.5,4.5) rectangle (5.5,5.5);

\draw[edge] (0,0) -- (0,6);
\draw[edge] (2,0) -- (2,6);
\draw[edge] (4,0) -- (4,6);
\draw[edge] (6,4) -- (6,6);
\draw[red, line width=6pt] (1,0) -- (1,3) -- (4,3) ;

\end{tikzpicture}
\captionsetup{labelformat=empty}	
\captionof{figure}{$(2,0,0)$}
\end{minipage}

\end{center}

\medskip
It follows from Theorem \ref{keylemma} that the weight filtration on $N(\a)$ is given by:
$$
\operatorname{gr}^W_{11}(N(\a))=L(3,2,2),\quad\quad \operatorname{gr}^W_{12}(N(\a))=L(2,2,2)\oplus L(3,2,1), 
$$
$$
\operatorname{gr}^W_{13}(N(\a))=L(3,0,0)\oplus L(2,2,1)\oplus L(1,0,0),\quad \quad \operatorname{gr}^W_{14}(N(\a))=L(2,0,0).
$$
The composition factors of $M(\a)$ are the same and $\operatorname{gr}_p^W(M(\a))=\operatorname{gr}^W_{-p}(N(\a))$ for all $p\in \mathbb{Z}$.\hfill\mbox{$\diamond$}		
\end{example}

For a larger example, see \cite[Example 2.19]{amy} for composition factors of $N(5,4,2,1)$ when $k=4$, $n= 9$, noting that loc. cit. uses the complementary convention for partitions. In particular, the simple $L(2,2)$ (shown in the bottom right) has weight $p=32$, which is maximal in $N(5,4,2,1)$.

Given $\mathbb{D}\in Dyck(\a)$, we write $v_{\a}^{\mathbb{D}}$ for any vector $v$ in $M(\a)$ such that the class of $v$ in a quotient of $M(\a)$ is a highest weight vector for $L(\a^{\mathbb{D}})$. Such vectors are called primitive vectors. For two Dyck patterns with no bullets $\mathbb{D},\mathbb{D}'\in Dyck(\a)$, we write $\mathbb{D}\subseteq \mathbb{D}'$ if all paths in $\mathbb{D}$ are also paths in $\mathbb{D}'$.

\begin{theorem}\label{thm:submod}
Let $\mathbb{D}$, $\mathbb{D}'\in Dyck(\a)$. The following is true about $M(\a)$.
\begin{enumerate}
\item If $\mathbb{D}\subseteq \mathbb{D}'$ then $v^{\mathbb{D}'}_{\a}\in U(\mf{g})	v_{\a}^{\mathbb{D}}$.
\item If $\mathbb{D}=\{(x,y)\}$ and $v^{\mathbb{D}'}_{\a}\in U(\mf{g})	v_{\a}^{\mathbb{D}}$, then $\mathbb{D}\subseteq \mathbb{D}'$.
\end{enumerate}
\end{theorem}

 We emphasize that the converse to (1) is false in general, see \cite[Theorem 5.18]{super} or \cite{beMorphisms}. For instance, in Example \ref{ex:comp}, a primitive vector for $L(2,0,0)$ generates a primitive vector for $L(1,0,0)$. In the terminology of \cite{super}, the Dyck pattern corresponding to $L(1,0,0)$ is a ``path with bridges".

\begin{proof}
The best reference for this information is obtained via \textit{super duality} (see, for example, \cite[Section 3.3]{super}) which relates parabolic Verma modules over an infinite-dimensional general linear Lie algebra to Kac modules over an infinite-dimensional general linear Lie superalgebra. Once we have set up notation, Theorem \ref{thm:submod} will follow readily from \cite[Theorem 5.18]{super}.

For each $r\geq n-k$ we consider the general linear Lie algebra $\mf{gl}_{k+r}$ with the parabolic subalgebra $\mf{p}_{k+r}$ with Levi subalgebra $\mf{gl}_{k}\times \mf{gl}_{r}$. Let $\a$ be as above, and consider the parabolic Verma module $M^{k+r}(\lambda_{\a})$ over $\mf{gl}_{k+r}$ with highest weight (see Section \ref{sec:young})
$$
\lambda_{\a,n}=(-a_k^c,-a_{k-1}^c,\cdots, -a_1^c\mid (a^c)_1',(a^c)_2',\cdots, (a^c)_r'),
$$ 
so that $M^{n}(\lambda_{\a,n})=M(\a)$.
 We consider the irreducible $\mf{gl}_{k+r}$-module $L^{k+r}(\lambda_{\a,n})$ of highest weight $\lambda_{\a,n}$, so that $L^{n}(\lambda_{\a,n})=L(\a)$. Taking the limit as $r\to \infty$, we obtain $\mf{gl}_{k+\infty}$ and the modules $M^{k+\infty}(\lambda_{\a,n})$, $L^{k+\infty}(\lambda_{\a,n })$. Write $\mc{O}^{\mf{p}_{k+r}}_{\tn{int}}$, $\mc{O}^{\mf{p}_{k+\infty}}_{\tn{int}}$ for the corresponding parabolic categories $\mc{O}$ with integral weights. Similar to Theorem \ref{keylemma}, $M^{k+\infty}(\lambda_{\a,n})$ has composition factors $L^{k+\infty}(-b_k,-b_{k-1},\cdots,-b_1|b_1',b_2',\cdots)$, where $\b=(b_1,\cdots,b_k)$ is obtained from $\a^c$ by \textit{adding} an admissible Dyck pattern with no bullets (over an infinite dimensional space, it is more convenient to work with $\a^c$ rather than $\a$). We have an exact truncation functor $\operatorname{tr}_{n-k}:\mc{O}^{\mf{p}_{k+\infty}}_{\tn{int}}\to \mc{O}^{\mf{p}_{n}}_{\tn{int}}$ sending $M^{k+\infty}(\lambda_{\a,n})$ to $M(\a)$ and $L^{k+\infty}(\lambda_{\a,n})$ to $L(\a)$. In particular, $\operatorname{tr}_{n-k}L^{k+\infty}(-b_k,-b_{k-1},\cdots,-b_1|b_1',b_2',\cdots)$ is nonzero if and only if $\b\in \mc{R}_{k,n}$, in which case $\operatorname{tr}_{n-k}L^{k+\infty}(-b_k,-b_{k-1},\cdots,-b_1|b_1',b_2',\cdots)=L(n-k-b_k,\cdots,n-k-b_1)=L(\b^c).$
 
 For $s\geq r$ we consider the general linear Lie superalgebra $\mf{gl}_{k|s}$ with parabolic subalgebra $\mf{p}_{k|s}$ and with Kac module $K^{k|s}(\lambda_{\a,n}^{\sharp})$ and irreducible module $L^{k|r}(\lambda_{\a,n}^{\sharp})$, where 
$$
\lambda^{\sharp}_{\a,n}=(-a_k^c,-a_{k-1}^c,\cdots, -a_1^c\mid a_1^c, a_2^c,\cdots,a_k^c,0^{s-k}).
$$
 We write $\mf{gl}_{k|\infty}$, $\mf{p}_{k|\infty}$, $K^{k|\infty}(\lambda_{\a,n}^{\sharp})$, and $L^{k|\infty}(\lambda_{\a,n}^{\sharp})$ for their limits, and we write $\mc{O}^{\mf{p}_{k|s}}_{\tn{int}}$, $\mc{O}^{\mf{p}_{k|\infty}}_{\tn{int}}$ for the corresponding parabolic categories $\mc{O}$. There is an exact truncation functor $\operatorname{tr}_s:\mc{O}^{\mf{p}_{k|\infty}}_{\tn{int}}\to \mc{O}^{\mf{p}_{k|s}}_{\tn{int}}$ sending $M^{k|\infty}(\lambda_{\a,n}^{\sharp})$ to $M^{k|s}(\lambda_{\a,n}^{\sharp})$ and $L^{k|\infty}(\lambda_{\a,n}^{\sharp})$ to $L^{k|s}(\lambda_{\a,n}^{\sharp})$. 

As a consequence of super duality (see \cite[Theorem 3.13]{super}), there is an equivalence of categories $\mc{O}^{\mf{p}_{k+\infty}}_{\tn{int}}\overset{\sim}\to \mc{O}^{\mf{p}_{k|\infty}}_{\tn{int}}$ identifying the pairs
$$
M^{k+\infty}(\lambda_{\a,n})\;\longleftrightarrow \;K^{k|\infty}(\lambda_{\a,n}^{\sharp}),\quad\quad L^{k+\infty}(\lambda_{\a,n})\;\longleftrightarrow \;L^{k|\infty}(\lambda_{\a,n}^{\sharp}).
$$
In particular, the composition factors of $K^{k|\infty}(\lambda_{\a,n}^{\sharp})$ are given by $L^{k|\infty}(-b_k,\cdots,-b_1|b_1,\cdots,b_k)$, where $\b=(b_1,\cdots, b_k)$ is obtained from $\a^c$ by adding an admissible Dyck pattern with no bullets. 
Given $\mathbb{D}\in Dyck(\a)$ we write $v_{\a}^{\mathbb{D}}$ for any vector in $M(\a)$ that becomes a highest weight vector for $L(\a^{\mathbb{D}})$ in a quotient of $M(\a)$, and we write $v_{\a}^{\mathbb{D},\infty}$ for any vector in $M^{k+\infty}(\lambda_{\a,n})$ that becomes a highest weight vector for $L^{k+\infty}(\lambda_{\a^{\mathbb{D}},n})$ in a quotient of $M^{k+\infty}(\lambda_{\a,n})$. Similarly, we write $\prescript{\sharp}{}{v_{\a}^{\mathbb{D},\infty}}$ for any vector in $K^{k|\infty}(\lambda_{\a,n}^{\sharp})$ that becomes a highest weight vector for $L^{k|\infty}(\lambda^{\sharp}_{\a^{\mathbb{D}},n})$ in a quotient of $K^{k|\infty}(\lambda_{\a,n}^{\sharp})$.

With this setup, we now proceed with the proof of Theorem \ref{thm:submod}.

(1) It suffices to show that $v^{\mathbb{D}'}_{\a}\in U(\mf{gl}_n)	v_{\a}^{\mathbb{D}}$ in $M(\a)$ when $\mathbb{D}\subseteq \mathbb{D}'$ and $|\mathbb{D}'|=|\mathbb{D}|+1$. Since the socle filtration on $M^{k+\infty}(\lambda_{\a,n})$ is determined by the number of paths in the Dyck pattern, it suffices to prove that $v^{\mathbb{D}',\infty}_{\a}\in U(\mf{gl}_{k+\infty})	v_{\a}^{\mathbb{D},\infty}$ in $M^{k+\infty}(\lambda_{\a,n})$. By super duality, this is equivalent to $\prescript{\sharp}{}{v^{\mathbb{D}',\infty}_{\a}}\in U(\mf{gl}_{k|\infty})	\prescript{\sharp}{}{v_{\a}^{\mathbb{D},\infty}}$ in $K^{k|\infty}(\lambda_{\a,n}^{\sharp})$, which follows from \cite[Theorem 5.18]{super}.

(2) Suppose that $\mathbb{D}=\{(x,y)\}$ and that $v^{\mathbb{D}'}_{\a}\in U(\mf{gl}_n)	v_{\a}^{\mathbb{D}}$ in $M(\a)$. By super duality we have that $\prescript{\sharp}{}{v^{\mathbb{D}',\infty}_{\a}}\in U(\mf{gl}_{k|\infty})	\prescript{\sharp}{}{v_{\a}^{\mathbb{D},\infty}}$ in $K^{k|\infty}(\lambda_{\a,n}^{\sharp})$. We induct on $|\mathbb{D}'|$. If $|\mathbb{D}'|=2$, then since $\mathbb{D}=\{(x,y)\}$ is a ``path without bridges", it follows from \cite[Theorem 5.18]{super} that $\mathbb{D}\subseteq \mathbb{D}'$. If $|\mathbb{D}'|>2$, then by inductive hypothesis there exists $\mathbb{D}''$ with $\mathbb{D}\subseteq \mathbb{D}''$, $\prescript{\sharp}{}{v^{\mathbb{D}'',\infty}_{\a}}\in U(\mf{gl}_{k|\infty})	\prescript{\sharp}{}{v_{\a}^{\mathbb{D},\infty}}$, $\prescript{\sharp}{}{v^{\mathbb{D}',\infty}_{\a}}\in U(\mf{gl}_{k|\infty})	\prescript{\sharp}{}{v_{\a}^{\mathbb{D}'',\infty}}$, and $|\mathbb{D}''|=|\mathbb{D}'|-1$. Then by \cite[Theorem 5.18]{super} we have that either $\mathbb{D}''\subseteq \mathbb{D}'$ or $\mathbb{D}'$ is obtained from $\mathbb{D}''$ by replacing a path $P\neq \{(x,y)\}$ in $\mathbb{D}''$ with two paths. In either case, we have that $\{(x,y)\} \in \mathbb{D}'$. 
\end{proof}

Let $\a,\b\in \mc{R}_{k,n}$ with $\b=\a^{(\{(x,y)\})}$, where $(x,y)$ is a corner of $\a$. By \cite{beMorphisms}, there is a unique (up to nonzero scalar) nonzero map $\phi^{\a}_{\b}:N(\a)\to N(\b)$. In particular, $\phi^{\a}_{\b}$ appears in the global Grothendieck--Cousin complex (Corollary \ref{globalGC}), and thus is strict with respect to the weight filtration $W_{\bullet}$.

\begin{corollary}\label{lem:keymorphismlemma}
Let $\a$ and $\b$ with $\b=\a^{(\{(x,y)\})}$, where $(x,y)$ is a corner of $\a$. Let $\phi^{\a}_{\b}$ be the natural map from $N(\a)$ to $N(\b)$. Let $p\in \mathbb{Z}$ and $\underline{c}$ be such that $L(\underline{c})$ is a composition factor of both $\operatorname{gr}^W_pN(\a)$ and $\operatorname{gr}^W_pN(\b)$, so that $\underline{c}=\a^{\mathbb{D}'}=\b^{\mathbb{D}}$ where $\mathbb{D}'$ is an $\a$-admissible Dyck pattern with $|\mathbb{D}'|=p+|\a|-2d_X$, and $\mathbb{D}$ is a $\b$-admissible Dyck pattern $|\mathbb{D}|=p+|\b|-2d_X$.

\begin{enumerate}
\item If $\mathbb{D}'=\mathbb{D}\cup (\{(x,y)\})$ then $(\operatorname{gr}^W_p\phi^{\a}_{\b})(L(\a^{\mathbb{D}'}))\cong L(\b^{\mathbb{D}})$.

\item If $\mathbb{D}'\neq \mathbb{D}\cup (\{(x,y)\})$ then $(\operatorname{gr}^W_p\phi^{\a}_{\b})(L(\a^{\mathbb{D}'}))=0$.
\end{enumerate}
\end{corollary}

\begin{proof}
We prove the dual statement for the natural map $\psi_{\a}^{\b}:M(\b)\to M(\a)$.

(1) The image of $\psi^{\a}_{\b}$ is the submodule of $M(\a)$ generated by $v_{\a}^{\{(x,y)\}}$. If $\mathbb{D}'$ is $\a$-admissible, then by Theorem \ref{thm:submod}(1) we have that $\psi^{\a}_{\b}(v_{\b}^{\mathbb{E}})=v_{\a}^{\mathbb{D}'}$ for some $\mathbb{E}$ with $\underline{b}^{\mathbb{E}}=\underline{c}$. Since $L(\underline{c})$ is a composition factor of $M(\b)$ with multiplicity one, we conclude that $\mathbb{E}=\mathbb{D}$, as required.

(2) If $\mathbb{D}'\neq \mathbb{D}\cup (\{(x,y)\})$ and $\psi^{\a}_{\b}(v_{\b}^{\mathbb{D}})\neq 0$, then $\psi^{\a}_{\b}(v_{\b}^{\mathbb{D}})=v_{\a}^{\mathbb{E}'}$ where $\{(x,y)\}\in \mathbb{E}'$ by Theorem \ref{thm:submod}(2). Since $\mathbb{D}'\neq \mathbb{E}'$ and $L(\underline{c})$ has multiplicity one as a composition factor of $M(\a)$, we obtain a contradiction.
\end{proof}

\subsection{Cohomology of the Grothendieck--Cousin complex}\label{sec:maincalc} The argument in this section is based on \cite[Section 3]{amy}, though functoriality and strictness of the weight filtration simplifies things in our setting (see Section \ref{amyremark}). We fix $\a\in \mc{R}_{k,n}$.

\begin{lemma}\label{lem:addingboxes}
Let $\b\subseteq \a$, and let $\mathbb{D}\in Dyck(\b)$. We define the following set of boxes
$$
\mathcal{I}(\a,\b,\mathbb{D})=\big\{(i,j)\in \a / \b \mid \tn{$\b\cup (i,j)$ is a partition, $\mathbb{D}\cup \{(i,j)\}\in Dyck(\b \cup (i,j))$}\big\}.
$$
 For all $J\subseteq 	\mathcal{I}(\a,\b,\mathbb{D})$ we have that $\b \cup J$ is a partition and $\mathbb{D}\cup J\in Dyck(\b\cup J)$.
\end{lemma}

\begin{proof}
Let $J\subseteq \mathcal{I}(\a,\b,\mathbb{D})$. If $\b \cup (i,j)$ is a partition for all $(i,j)\in J$, then no box in $J$ lies directly North, South, East, nor West of another. Thus, $b\cup J$ is a partition as well. Since no box in $J$ lies directly North, South, East, nor West of another, and $\mathbb{D}\cup\{(i,j)\}$ is an admissible Dyck pattern for all $(i,j)\in J$, it follows that $\mathbb{D}\cup J$ is an admissible Dyck pattern.
\end{proof}

Let $GC^{\bullet}_{\a}$ denote the global Grothendieck--Cousin complex for $\a$, so that for $j,p\in \mathbb{Z}$ we have 
\begin{equation}\label{decompsemisimple}
GC_{\a}^j=\bigoplus_{\b\subseteq \a,\;|\b|=d_X-j}N(\b),\quad\quad \operatorname{gr}^W_p(GC^j_{\a})=\bigoplus_{\b\subseteq \a,\;|\b|=d_X-j} \bigoplus_{\mathbb{D}\in \mathcal{Z}_p(\b)} L(\b^{\mathbb{D}}).
\end{equation}
Using Corollary \ref{lem:keymorphismlemma} we break $\operatorname{gr}^W_p GC^{\bullet}_{\a}$ into a direct sum of Koszul subcomplexes.

\begin{theorem}\label{Koszulsums}
Given $p,q\in \mathbb{Z}$ we define
$$
\mathcal{Y}_p(\a;q)=\left\{(\b,\mathbb{D})\middle\vert \begin{array}{l}
    \tn{$b\subseteq \a$ and  $q=d_{X}-|\b|$, }\\
    \tn{$\mathbb{D}\in Dyck(\b)$,} \\
    \tn{$|D|\geq 3$ for all $D\in \mathbb{D}$ and $|\mathbb{D}|=p-q-d_{X}$ }
    \end{array} \right\}.
$$
The following is true about the complex $\operatorname{gr}^W_pGC^{\bullet}_{\a}$.
\begin{enumerate}
\item We have that $\operatorname{gr}^W_pGC^{\bullet}_{\a}$ is a direct sum of subcomplexes
$$
\operatorname{gr}^W_p\big(GC^{\bullet}_{\a}\big)=\bigoplus_{0\leq q\leq |\a|}\bigoplus_{(\b,\mathbb{D})\in \mathcal{Y}_p(\a;q)} \mathcal{K}(\a,\b,\mathbb{D}),
$$
where $\mathcal{K}(\a,\b,\mathbb{D})$ is a complex concentrated in degrees $q-|\mathcal{I}(\a,\b,\mathbb{D})|,\cdots , q$, with
$$
\mathcal{K}(\a,\b,\mathbb{D})_j=\bigoplus_{J\subseteq \mathcal{I}(\a,\b,\mathbb{D}),\; |J|=q-j} L((\b\cup J)^{\mathbb{D}\cup J}),\quad q-|\mathcal{I}(\a,\b,\mathbb{D})|\leq j \leq q,
$$

\item For $0\leq q\leq d_{X}$ we have
$$
\mathcal{H}^q\big(\operatorname{gr}^W_p\big(GC^{\bullet}_{\a}\big)\big) =\bigoplus_{\substack{(\b,\mathbb{D})\in \mathcal{Y}_p(\a;q)\\ \mathcal{I}(\a,\b,\mathbb{D})=\emptyset}} L(\b^{\mathbb{D}}).
$$	
\end{enumerate}

\end{theorem}
We note that all summands of $\mathcal{K}(\a,\b,\mathbb{D})$ are isomorphic to $L(\b^{\mathbb{D}})$. 

\begin{proof}
Let $0\leq j\leq d_{X}$ and let $N(\underline{x})$ be a summand of $GC^j_{\a}$. Let $\mathbb{X}\in Dyck(\underline{x})$ with $|\mathbb{X}|=p-j-d_{X}$, so that $L(\underline{x}^{\mathbb{X}})$ is a composition factor of $\operatorname{gr}^W_pN(\underline{x})$. Let $J\subseteq \underline{x}$ be the set of boxes in $\underline{x}$ corresponding to paths of length one in $\mathbb{X}$. Set $\underline{b}=\underline{x}\setminus J$ and $\mathbb{D}=\mathbb{X}\setminus J$. Then $(\underline{b},\mathbb{D})\in \mathcal{Y}_p(\a;j+|J|)$. Thus, $L(\underline{x}^{\mathbb{X}})$ is a summand of $\mathcal{K}(\a,\b,\mathbb{D})_j$, and there is no other pair $(\underline{b}',\mathbb{D}')\in\mathcal{Y}_p(\a;j+|J|)$ for which $L(\underline{x}^{\mathbb{X}})$ is a summand of $\mathcal{K}(\a,\b',\mathbb{D}')_j$.

By Corollary \ref{globalGC}, the morphism $\partial^j:GC^j_{\a}\to GC^{j+1}_{\a}$ induces a nonzero morphism from $N(\underline{x})$ to $N(\underline{x}\setminus k)$ for all $k\in J$ (and such a morphism is unique up to a nonzero scalar \cite{beMorphisms}). By Corollary \ref{lem:keymorphismlemma} it follows that $\partial^j$ maps $L(\underline{x}^{\mathbb{X}})$ diagonally to
$$
\bigoplus_{k\in J} L((\underline{x}\setminus k)^{\mathbb{X}\setminus k}).
$$
This proves (1). To see (2), note that $\mathcal{K}(\a,\b,\mathbb{D})$ is a Koszul complex (see \cite[Lemma 3.6]{amy}) and hence is exact unless it has only one term, which is equivalent to $\mathcal{I}(\a,\b,\mathbb{D})=\emptyset$.
\end{proof}

\begin{example}
Let $X=\operatorname{Gr}(2,4)$ and $\a=(2,1)$. Then $GC^{\bullet}_{\a}$ is the complex
$$
GC^{\bullet}_{\a}:\quad \quad 0\longrightarrow 0\longrightarrow \textnormal{$\large{N}\left(\tiny{\yng(2,1)}\right)$}\longrightarrow \textnormal{$\large{N}\left(\tiny{\yng(2)}\right)\oplus \large{N}\left(\tiny{\yng(1,1)}\right)$}\longrightarrow \textnormal{$\large{N}\left(\tiny{\yng(1)}\right)$}\longrightarrow \textnormal{$\large{N}\left(0\right)$}\longrightarrow 0.
$$
In the following table, the column indices are the cohomological degree and the row indices indicate the graded pieces of the weight filtration:\\
\begin{center}
\scalebox{1.0}{
{\tabulinesep=1.7mm
\begin{tabu}{|c |  c | c | c | c |} 
 \hline

  &  $1$ & $2$ & $3$ & $4$ \\ 
  \hline

 $5$ &  $\large{L}\left(\tiny{\yng(2,1)}\right)$ &  &  &  \\  
 \hline
 
  $6$ &  $\quad\large{L}\left(\tiny{\yng(2)}\right)\oplus \large{L}\left(\tiny{\yng(1,1)}\right)\oplus \large{L}\left(0\right)\quad$ & $\large{L}\left(\tiny{\yng(2)}\right)\oplus \large{L}\left(\tiny{\yng(1,1)}\right)$ &  &  \\  
 \hline
 
  $7$ &  $\large{L}\left(\tiny{\yng(1)}\right)$ & $\quad\large{L}\left(\tiny{\yng(1)}\right)\oplus \large{L}\left(\tiny{\yng(1)}\right)\quad$ & $\large{L}\left(\tiny{\yng(1)}\right)$ &  \\  
 \hline
 
  $8$ &   &  & $\quad L\left(0\right)\quad$ & $\quad L\left(0\right)\quad$ \\  
 \hline
\end{tabu}}}
\end{center}

\medskip

\noindent To create the sets $\mathcal{Y}_p(\a;q)$ we select all composition factors of $N(\b)$ in cohomological degree $q$ that were obtained by removing Dyck patterns for which all paths have length $\geq 3$ (including the empty pattern):
$$
\mathcal{Y}_5(\a;1)=\left\{\left(\tiny{\yng(2,1)},\emptyset\right) \right\},\quad \mathcal{Y}_6(\a;1)=\left\{\left(\tiny{\yng(2,1)},\tiny{\yng(2,1)}\right) \right\},\quad \mathcal{Y}_6(\a;2)=\left\{\left(\tiny{\yng(2)},\emptyset\right),\left(\tiny{\yng(1,1)},\emptyset\right) \right\},
$$
$$
\mathcal{Y}_7(\a;3)=\left\{\left(\tiny{\yng(1)},\emptyset\right) \right\},\quad \mathcal{Y}_8(\a;4)=\left\{\left((0),\emptyset\right) \right\}.
$$
Therefore, we obtain direct sum decompositions:
$$
\operatorname{gr}_5^W(GC^{\bullet}_{\a})=\mc{K}\left(\a, \tiny{\yng(2,1)},\emptyset\right),
$$
$$
\operatorname{gr}_6^W(GC^{\bullet}_{\a})=\mc{K}\left(\a, \tiny{\yng(2,1)},\tiny{\yng(2,1)}\right)\oplus \mc{K}\left(\a, \tiny{\yng(2)},\emptyset\right)\oplus \mc{K}\left(\a, \tiny{\yng(1,1)},\emptyset\right),
$$
$$
\operatorname{gr}_7^W(GC^{\bullet}_{\a})=\mc{K}\left(\a, \tiny{\yng(1)},\emptyset\right),\quad \quad 
\operatorname{gr}_8^W(GC^{\bullet}_{\a})=\mc{K}\left(\a, (0),\emptyset\right).
$$
For each pair $(\b,\mathbb{D})$ as above, we have $\mathcal{I}(\a,\b,\mathbb{D})=\emptyset$ if and only if $(\b,\mathbb{D})$ is one of the following:
$$
\left(\tiny{\yng(2,1)},\emptyset\right)\quad \tn{and} \quad \left(\tiny{\yng(2,1)},\tiny{\yng(2,1)}\right).
$$
These correspond to the Koszul complexes with one term.
In particular, we obtain a non-split exact sequence
$$
0\longrightarrow \textnormal{$\large{L}\left(\tiny{\yng(2,1)}\right)$}\longrightarrow H^1_{Z_{\a}}(X,\mc{O}_X)\longrightarrow \textnormal{$\large{L}\left(0\right)$}\longrightarrow 0,
$$
and $H^q_{Z_{\a}}(X,\mc{O}_X)=0$ for $q\neq 1$. We note that $Z_{(0)}$ is the singular locus of $Z_{\a}$.
\hfill\mbox{$\diamond$}		
\end{example}

We provide a combinatorial bijection relating Theorem \ref{Koszulsums} to Theorem \ref{mainthm}.
\begin{lemma}\label{finallemma}
There is a bijection between the sets
$$
\{(\b,\mathbb{D})\in \mathcal{Y}_p(\a;q)\mid \mathcal{I}(\a,\b,\mathbb{D})=\emptyset\}\quad \leftrightarrow \quad \mathcal{A}_p(\aa;q),
$$	
given by sending $(\b,\mathbb{D}=(D_1,\cdots,D_{p-q-d_{X}}))$ to
$$
\mathbb{D}'=(D_1,\cdots,D_{p-q-d_{X}};\mathbb{B}),
$$
where $\mathbb{B}$ is the set of boxes in $\a /\b$.
\end{lemma}

\begin{proof}
We note that the number of paths and bullets in $\mathbb{D}'$ is correct. To show that the map is well-defined, we need to verify that $\mathbb{D}'\in Dyck^{\bullet}(\a)$. First, $\a^{\mathbb{D}'}=\b^{\mathbb{D}}$ is a partition, as $\mathbb{D}$ is $\b$-admissible. Since $\mathbb{D}$ is a Dyck pattern, we have that $(D_1,\cdots,D_{p-q-d_{X}})$ satisfies the covering condition. Since the support of $\mathbb{D}$ is contained in $\b$, and the support of $\mathbb{B}$ is $\a/\b$, it follows that no box of $\mathbb{B}$ is located directly North, Northwest, nor West from a box in $\mathbb{D}$. Therefore, the map is well-defined.

The map is clearly injective, as $\mathbb{B}$ completely determines $\b$, and $(D_1,\cdots,D_{p-q-d_{X}})$ completely determines $\mathbb{D}$. To show surjectivity, let $\mathbb{D}'=(D_1,\cdots,D_{p-q-d_{X}};\mathbb{B})\in \mathcal{A}_p(\aa;q)$, $\b=\a\setminus \mathbb{B}$, and $\mathbb{D}=(D_1,\cdots,D_{p-q-d_{X}})$. We need to show that $\mathcal{I}(\a,\b,\mathbb{D})=\emptyset$. All boxes of $\a/\b$ belong to the heads or tails of paths in $\mathbb{D}$. Thus, for every box $(i,j)$ in $\mathbb{B}$ such that $\b\cup (i,j)$ is a partition, there exists a path $D\in \mathbb{D}$ such that $(i,j)$ belongs directly to the head or tail of $D$. In particular, the covering condition fails for $D$ and $(i,j)$. Thus, $(i,j)$ does not belong to $\mathcal{I}(\a,\b,\mathbb{D})$, and thus $\mathcal{I}(\a,\b,\mathbb{D})=\emptyset$.
\end{proof}

Combining Corollary \ref{globalGC}, Theorem \ref{Koszulsums}(2), Lemma \ref{finallemma}, we obtain the descriptions of $[\mathscr{H}^q_{Z_{\underline{a}}}(\mc{O}_X)]$ and $[H^q_{Z_{\underline{a}}}(X,\mc{O}_X)]$ in Theorem \ref{mainthm}. Via Theorem \ref{GCaComplex}, the following result implies the assertion in Theorem \ref{mainthm} about the graded pieces of the weight filtration.

\begin{corollary}
For all $\underline{a}\in \mathcal{R}_{k,n}$, $0\leq q\leq d_X$, and $p\in \mathbb{Z}$ we have that 
$$
\mathcal{H}^q(\operatorname{gr}^W_p\mathcal{GC}^{\bullet}_{\a})\cong \bigoplus_{\mathbb{D}\in \mathcal{A}_p(\aa;q)} IC^H_{Z_{\a^{\mathbb{D}}}}\left(\frac{\left(\left|\a^{\mathbb{D}}\right|-p\right)}{2}\right).
$$
\end{corollary}

\begin{proof}
By multiplicity-freeness of the modules $N(\underline{b})$ (Theorem \ref{keylemma}), we have that, for all $\b\in \mathcal{R}_{k,n}$ and all $p\in \mathbb{Z}$, the Hodge modules $\operatorname{gr}^W_p( \mathscr{H}^{c(\b)}_{O_{\b}}(\mathcal{O}_X^H))$ are direct sums of intersection cohomology Hodge modules of the form $IC^H_{Z_{\b^{\mathbb{D}}}}(k)$, where $k=(|\underline{b}^{\mathbb{D}}|-p)/2$, so that the summands have weight $p$. It follows that the direct sum decompositions in (\ref{decompsemisimple}) and the conclusions of Theorem \ref{Koszulsums} may be lifted to the category of mixed Hodge modules. In particular, the complex $\operatorname{gr}^W_p\mathcal{GC}^{\bullet}_{\a}$ of Hodge modules is a direct sum of Koszul subcomplexes of intersection cohomology Hodge modules of the form $IC^H_{Z_{\b^{\mathbb{D}}}}(k)$, and $\mathcal{H}^q(\operatorname{gr}^W_p\mathcal{GC}^{\bullet}_{\a})$ is a direct sum over $\mathbb{D}\in \mathcal{A}_p(\aa;q)$ of modules $IC^H_{Z_{\a^{\mathbb{D}}}}(k)$, where $k=(|\a^{\mathbb{D}}|-p)/2$, so that the summands have weight $p$.
\end{proof}

\begin{remark}
One could attempt to carry out a similar argument for the remaining Hermitian pairs:	
$$
(\mathsf{B}_n,\mathsf{B}_{n-1}),\quad (\mathsf{C}_n,\mathsf{A}_{n-1}),\quad (\mathsf{D}_n,\mathsf{A}_{n-1}),\quad (\mathsf{D}_n,\mathsf{D}_{n-1}),\quad (\mathsf{E}_6,\mathsf{D}_5),\quad (\mathsf{E}_7,\mathsf{E}_6).
$$
The corresponding parabolic Verma modules are again rigid and multiplicity-free \cite{cisWeight}, and the composition factors are described by parabolic Kazhdan--Lusztig polynomials \cite{CC} (see also \cite{Boe, brenti}). Unfortunately, the submodule lattices are not known in full generality, so that we do not currently have analogues of Theorem \ref{thm:submod} for all Hermitian pairs. One may use our techniques in any setting where the relevant parabolic Verma modules have Loewy length at most three (because then the submodule lattice is determined by the Loewy filtration), which covers the cases $(\mathsf{B}_n,\mathsf{B}_{n-1})$ and $(\mathsf{D}_n,\mathsf{D}_{n-1})$. In the cases $(\mathsf{E}_6,\mathsf{D}_5)$, $(\mathsf{E}_7,\mathsf{E}_6)$, many of the parabolic Verma modules have Loewy length at most three (see \cite[Tables 7.1, 7.2]{cisWeight}), so our techniques are applicable to some of the Schubert varieties. In the major cases $(\mathsf{C}_n,\mathsf{A}_{n-1})$, $(\mathsf{D}_n,\mathsf{A}_{n-1})$, the weight filtrations are dictated by shifted Dyck patterns \cite{brenti}. We expect a combinatorial formula for local cohomology in these cases similar to ours, via ``augmented shifted Dyck patterns". Formulating and proving such a result would allow one to calculate the weight filtration on local cohomology with support in determinantal varieties of symmetric matrices and Pfaffian varieties of skew-symmetric matrices as in Section \ref{sec:det}, upon restricting to the opposite big cell. The composition factors are known \cite{raicu2016local}, but the Hodge and weight filtrations are not.  
\end{remark}

\subsection{Equivalence to syzygies of determinantal thickenings}\label{amyremark}

A determinantal thickening is a $\GL_m(\C)\times \GL_n(\C)$ invariant ideal in the polynomial ring $S=\operatorname{Sym}(\C^m\otimes \C^n)\cong \C[x_{i,j}]_{1\leq i\leq m,1\leq j\leq n}$ on the space of $m\times n$ matrices, with $m\geq n$ (see \cite{DEP}). Given a partition $\lambda=(\lambda_1\geq \cdots \geq \lambda_n\geq 0)$, there is a basic thickening $I_{\lambda}\subseteq S$ generated by the subrepresentation $\bS_{\lambda}\C^m\otimes \bS_{\lambda}\C^n\subseteq S$. Via the BGG correspondence, the linear strands of the free resolutions of all determinantal thickenings admit actions of the general linear Lie superalgebra $\mathfrak{gl}(m|n)$ (see \cite{thick} and references therein). In \cite{thick}, Raicu--Weyman make a conjecture on the structure of the linear strands of the basic thickenings $I_{\lambda}$, in terms of admissible augmented Dyck patterns \cite[Conjecture 4.1]{thick}, and prove it for the first strand \cite[Theorem 5.1]{thick}. Huang proves their full conjecture in \cite[Theorem 3.7]{amy}.

Via super duality (see \cite[Section 3.3]{super} or the proof of Theorem \ref{thm:submod} above), our Theorem \ref{mainthm} is equivalent to \cite[Theorem 3.7]{amy}. Indeed, over infinite-dimensional Lie (super)algebras, our complex $GC^{\bullet}_{\a}$ is identified with dual of the complex $\tilde{\mathbf{R}}(I_{\a^c})$ in \cite{amy}. The cohomology of $GC^{\bullet}_{\a}$ calculates local cohomology for us, while the cohomology of $\tilde{\mathbf{R}}(I_{\a^c})$ calculates the linear strands of the free resolution of $I_{\a^c}$. This explains the appearance of admissible augmented Dyck patterns in both works. Our argument above is closely modeled after Huang's. There are two advantages to our treatment, however. Firstly, from our point of view, it is clear that the weight filtration is functorial and that 	$GC^{\bullet}_{\a}$ is a filtered complex, whereas it is not immediate that the Loewy filtration in \cite{amy} is compatible with the morphisms in $\tilde{\mathbf{R}}(I_{\a^c})$. Secondly, strictness of the weight filtration immediately implies that the spectral sequence for $(GC^{\bullet}_{\a}, W_{\bullet})$ degenerates, whereas it requires justification for $\tilde{\mathbf{R}}(I_{\a^c})$.

It would be interesting to use duality results for other Lie (super)algebras (see \cite{dualities}) to relate local cohomology on other flag varieties to linear strands of free resolutions of interesting classes of ideals. A natural case to investigate next would be the Hermitian pairs $(\mathsf{C}_n,\mathsf{A}_{n-1})$, $(\mathsf{D}_n,\mathsf{A}_{n-1})$.

\subsection{Application to local cohomology supported in determinantal varieties}\label{sec:det}
Let $\mathscr{X}=\C^{m\times n}$ be the space of $m\times n$ matrices over the complex numbers with $m\geq n$, and for $0\leq p \leq n$ let $\mathscr{Z}_p\subseteq \X$ denote the determinantal variety of matrices of rank $\leq p$. In recent past, a great deal of information has been learned about the local cohomology modules $H^j_{\Z_p}(\X,\mc{O}_{\X})$, see \cite{witt, rw, characters, raicu2016local, categories, iterated, perlmanraicu, MHM}. 

In this section, we use Theorem \ref{mainthm} to recover the composition factors and weight filtration on $H^j_{\Z_p}(\X,\mc{O}_{\X})$, obtained in full generality in \cite{raicu2016local} and \cite{MHM}. We set $D_p=\mathcal{L}(\Z_p,\X)$ for the intersection cohomology $\D$-module associated to the trivial local system on $\Z_p\setminus \Z_{p-1}$. 
\begin{theorem}\label{thm:detthm}
The following is true about the local cohomology modules $H^j_{\Z_p}(\X,\mc{O}_{\X})$:
\begin{enumerate}
\item \cite[Main Theorem]{raicu2016local} We have the following in the Grothendieck group of $\D$-modules on $\X$:
$$
\sum_{j\geq 0} \big[H^j_{\Z_p}(\X,\mc{O}_{\X})\big]\cdot q^j=\sum_{s=0}^p [D_s]\cdot q^{(n-p)^2+(n-s)\cdot (m-n)}\cdot \binom{n-s-1}{p-s}_{q^2}.
$$

\item \cite[Theorem 1.1]{MHM} Each copy of $D_s$ in $H^j_{\Z_p}(\X,\mc{O}_{\X})$ underlies a pure Hodge module of weight $mn+p-s+j$.
\end{enumerate}
\end{theorem}
Let $X=\operatorname{Gr}(n,m+n)$	 and let $U\subseteq X$ be the opposite big cell with respect to the fixed flag $\mathcal{V}_{\bullet}$. Then $U\cong \X$ and via this identification we have that $\Sv{\underline{p}}|_U=\Z_p$,
where $\underline{p}=(m^p,p^{n-p})$ is the partition obtained from the $n\times m$ rectangle by removing the bottom-right $(n-p)\times (m-p)$ rectangle (see \cite[Chapter 5]{standard}). For instance, when $m=5$ and $n=4$, the following Young diagrams correspond to Schubert varieties in $\operatorname{Gr}(4,9)$ that restrict to the determinantal varieties in $\C^{5\times 4}$:\\

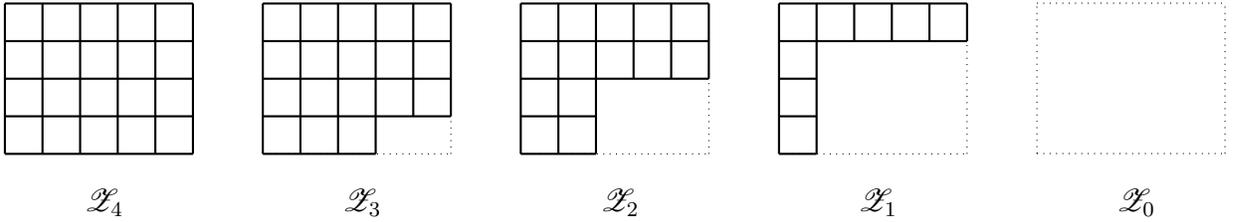
\begin{center}
\begin{minipage}{.18\textwidth}
\centering 
\begin{tikzpicture}[x=\unitsize,y=\unitsize,baseline=0]
\tikzset{vertex/.style={}}%
\tikzset{edge/.style={  thick}}%
\draw[edge] (0,2) -- (10,2);
\draw[edge] (0,4) -- (10,4);
\draw[edge] (0,6) -- (10,6);
\draw[edge] (0,8) -- (10,8);
\draw[edge] (0,10) -- (10,10);
\draw[edge] (0,2) -- (0,10);
\draw[edge] (2,2) -- (2,10);
\draw[edge] (4,2) -- (4,10);
\draw[edge] (6,2) -- (6,10);
\draw[edge] (8,2) -- (8,10);
\draw[edge] (10,2) -- (10,10);
\end{tikzpicture}
\vspace{-.5cm}
\captionsetup{labelformat=empty}	
\captionof{figure}{$\Z_4$}	
\end{minipage}
\begin{minipage}{.18\textwidth}
\centering 
\begin{tikzpicture}[x=\unitsize,y=\unitsize,baseline=0]
\tikzset{vertex/.style={}}%
\tikzset{edge/.style={  thick}}%
\draw[edge] (0,2) -- (6,2);
\draw[edge] (0,4) -- (10,4);
\draw[edge] (0,6) -- (10,6);
\draw[edge] (0,8) -- (10,8);
\draw[edge] (0,10) -- (10,10);
\draw[edge] (0,2) -- (0,10);
\draw[edge] (2,2) -- (2,10);
\draw[edge] (4,2) -- (4,10);
\draw[edge] (6,2) -- (6,10);
\draw[edge] (8,4) -- (8,10);
\draw[edge] (10,4) -- (10,10);
\draw[dotted] (10,2) -- (10,10);
\draw[dotted] (0,2) -- (10,2);
\end{tikzpicture}
\vspace{-.5cm}
\captionsetup{labelformat=empty}	
\captionof{figure}{$\Z_3$}	
\end{minipage}
\begin{minipage}{.18\textwidth}
\centering 
\begin{tikzpicture}[x=\unitsize,y=\unitsize,baseline=0]
\tikzset{vertex/.style={}}%
\tikzset{edge/.style={  thick}}%
\draw[edge] (0,2) -- (4,2);
\draw[edge] (0,4) -- (4,4);
\draw[edge] (0,6) -- (10,6);
\draw[edge] (0,8) -- (10,8);
\draw[edge] (0,10) -- (10,10);
\draw[edge] (0,2) -- (0,10);
\draw[edge] (2,2) -- (2,10);
\draw[edge] (4,2) -- (4,10);
\draw[edge] (6,6) -- (6,10);
\draw[edge] (8,6) -- (8,10);
\draw[edge] (10,6) -- (10,10);
\draw[dotted] (10,2) -- (10,10);
\draw[dotted] (0,2) -- (10,2);
\end{tikzpicture}
\vspace{-.5cm}
\captionsetup{labelformat=empty}	
\captionof{figure}{$\Z_2$}	
\end{minipage}
\begin{minipage}{.18\textwidth}
\centering 
\begin{tikzpicture}[x=\unitsize,y=\unitsize,baseline=0]
\tikzset{vertex/.style={}}%
\tikzset{edge/.style={  thick}}%
\draw[edge] (0,2) -- (2,2);
\draw[edge] (0,4) -- (2,4);
\draw[edge] (0,6) -- (2,6);
\draw[edge] (0,8) -- (10,8);
\draw[edge] (0,10) -- (10,10);
\draw[edge] (0,2) -- (0,10);
\draw[edge] (2,2) -- (2,10);
\draw[edge] (4,8) -- (4,10);
\draw[edge] (6,8) -- (6,10);
\draw[edge] (8,8) -- (8,10);
\draw[edge] (10,8) -- (10,10);
\draw[dotted] (10,2) -- (10,10);
\draw[dotted] (0,2) -- (10,2);
\end{tikzpicture}	
\vspace{-.5cm}
\captionsetup{labelformat=empty}	
\captionof{figure}{$\Z_1$}
\end{minipage}
\begin{minipage}{.18\textwidth}
\centering 
\begin{tikzpicture}[x=\unitsize,y=\unitsize,baseline=0]
\tikzset{vertex/.style={}}%
\tikzset{edge/.style={  thick}}%
\draw[dotted] (0,10) -- (10,10);
\draw[dotted] (0,2) -- (0,10);
\draw[dotted] (10,2) -- (10,10);
\draw[dotted] (0,2) -- (10,2);
\end{tikzpicture}
\vspace{-.5cm}
\captionsetup{labelformat=empty}	
\captionof{figure}{$\Z_0$}
\end{minipage}
\end{center}

\medskip

Furthermore, via this identification, we have that 
\begin{equation}
\mathcal{L}(\Sc{\underline{p}},X)|_U=D_p,\quad IC^H_{\Sv{\underline{p}}}|_U\cong IC^H_{\Z_p},\quad \mathscr{H}^j_{\Sv{\underline{p}}}(\mathcal{O}_X^H)|_U\cong H^j_{\Z_p}(\X,\mc{O}_{\X}^H).
\end{equation}

We observe that the only elements of $Dyck^{\bullet}(\underline{p})$ consisting of paths of length $\geq 3$ are obtained by removing hooks supported on the inside corner of $\underline{p}$. For instance, continuing the above example, we obtain the following elements of $Dyck^{\bullet}(\underline{2})$ when $m=5$ and $n=4$:\\

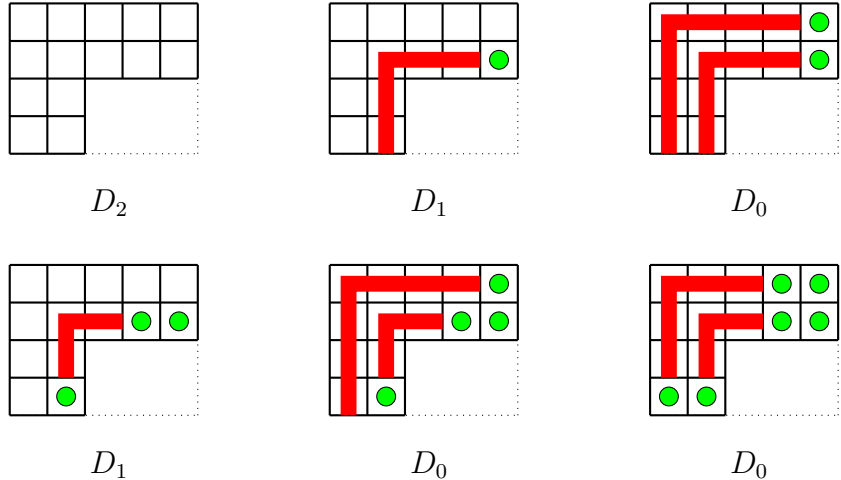
\begin{center}
\begin{minipage}{.18\textwidth}
\centering 
\begin{tikzpicture}[x=\unitsize,y=\unitsize,baseline=0]
\tikzset{vertex/.style={}}%
\tikzset{edge/.style={  thick}}%
\draw[edge] (0,2) -- (4,2);
\draw[edge] (0,4) -- (4,4);
\draw[edge] (0,6) -- (10,6);
\draw[edge] (0,8) -- (10,8);
\draw[edge] (0,10) -- (10,10);
\draw[edge] (10,6) -- (10,10);
\draw[edge] (0,2) -- (0,10);
\draw[edge] (2,2) -- (2,10);
\draw[edge] (4,2) -- (4,10);
\draw[edge] (6,6) -- (6,10);
\draw[edge] (8,6) -- (8,10);
\draw[dotted] (10,2) -- (10,10);
\draw[dotted] (0,2) -- (10,2);
\end{tikzpicture}
\vspace{-.5cm}
\captionsetup{labelformat=empty}	
\captionof{figure}{$D_2$}	
\end{minipage}
\quad\quad
\begin{minipage}{.18\textwidth}
\centering 
\begin{tikzpicture}[x=\unitsize,y=\unitsize,baseline=0]
\tikzset{vertex/.style={}}%
\tikzset{edge/.style={  thick}}%
\draw[edge] (0,2) -- (4,2);
\draw[edge] (0,4) -- (4,4);
\draw[edge] (0,6) -- (10,6);
\draw[edge] (0,8) -- (10,8);
\draw[edge] (0,10) -- (10,10);
\draw[edge] (0,2) -- (0,10);
\draw[edge] (2,2) -- (2,10);
\draw[edge] (4,2) -- (4,10);
\draw[edge] (6,6) -- (6,10);
\draw[edge] (8,6) -- (8,10);
\draw[edge] (10,6) -- (10,10);
\draw[red, line width=6pt]  (3,2) -- (3,7) -- (8,7);
\draw[fill=green] (9,7) circle [radius=0.5];
\draw[dotted] (10,2) -- (10,10);
\draw[dotted] (0,2) -- (10,2);
\end{tikzpicture}
\vspace{-.5cm}
\captionsetup{labelformat=empty}	
\captionof{figure}{$D_1$}
\end{minipage}
\quad\quad
\begin{minipage}{.18\textwidth}
\centering 
\begin{tikzpicture}[x=\unitsize,y=\unitsize,baseline=0]
\tikzset{vertex/.style={}}%
\tikzset{edge/.style={  thick}}%
\draw[edge] (0,2) -- (4,2);
\draw[edge] (0,4) -- (4,4);
\draw[edge] (0,6) -- (10,6);
\draw[edge] (0,8) -- (10,8);
\draw[edge] (0,10) -- (10,10);
\draw[edge] (0,2) -- (0,10);
\draw[edge] (2,2) -- (2,10);
\draw[edge] (4,2) -- (4,10);
\draw[edge] (6,6) -- (6,10);
\draw[edge] (8,6) -- (8,10);
\draw[edge] (10,6) -- (10,10);
\draw[red, line width=6pt]  (3,2) -- (3,7) -- (8,7);
\draw[red, line width=6pt]  (1,2) -- (1,9) -- (8,9);
\draw[fill=green] (9,7) circle [radius=0.5];
\draw[fill=green] (9,9) circle [radius=0.5];
\draw[dotted] (10,2) -- (10,10);
\draw[dotted] (0,2) -- (10,2);
\end{tikzpicture}
\vspace{-.5cm}
\captionsetup{labelformat=empty}	
\captionof{figure}{$D_0$}
\end{minipage}

\medskip
\medskip
\begin{minipage}{.18\textwidth}
\centering 
\begin{tikzpicture}[x=\unitsize,y=\unitsize,baseline=0]
\tikzset{vertex/.style={}}%
\tikzset{edge/.style={  thick}}%
\draw[edge] (0,2) -- (4,2);
\draw[edge] (0,4) -- (4,4);
\draw[edge] (0,6) -- (10,6);
\draw[edge] (0,8) -- (10,8);
\draw[edge] (0,10) -- (10,10);
\draw[edge] (0,2) -- (0,10);
\draw[edge] (2,2) -- (2,10);
\draw[edge] (4,2) -- (4,10);
\draw[edge] (6,6) -- (6,10);
\draw[edge] (8,6) -- (8,10);
\draw[edge] (10,6) -- (10,10);
\draw[red, line width=6pt]  (3,4) -- (3,7) -- (6,7);
\draw[fill=green] (3,3) circle [radius=0.5];
\draw[fill=green] (7,7) circle [radius=0.5];
\draw[fill=green] (9,7) circle [radius=0.5];
\draw[dotted] (10,2) -- (10,10);
\draw[dotted] (0,2) -- (10,2);
\end{tikzpicture}
\vspace{-.5cm}
\captionsetup{labelformat=empty}	
\captionof{figure}{$D_1$}	
\end{minipage}
\quad\quad
\begin{minipage}{.18\textwidth}
\centering 
\begin{tikzpicture}[x=\unitsize,y=\unitsize,baseline=0]
\tikzset{vertex/.style={}}%
\tikzset{edge/.style={  thick}}%
\draw[edge] (0,2) -- (4,2);
\draw[edge] (0,4) -- (4,4);
\draw[edge] (0,6) -- (10,6);
\draw[edge] (0,8) -- (10,8);
\draw[edge] (0,10) -- (10,10);
\draw[edge] (0,2) -- (0,10);
\draw[edge] (2,2) -- (2,10);
\draw[edge] (4,2) -- (4,10);
\draw[edge] (6,6) -- (6,10);
\draw[edge] (8,6) -- (8,10);
\draw[edge] (10,6) -- (10,10);
\draw[red, line width=6pt]  (3,4) -- (3,7) -- (6,7);
\draw[red, line width=6pt]  (1,2) -- (1,9) -- (8,9);
\draw[fill=green] (3,3) circle [radius=0.5];
\draw[fill=green] (7,7) circle [radius=0.5];
\draw[fill=green] (9,7) circle [radius=0.5];
\draw[fill=green] (9,9) circle [radius=0.5];
\draw[dotted] (10,2) -- (10,10);
\draw[dotted] (0,2) -- (10,2);
\end{tikzpicture}
\vspace{-.5cm}
\captionsetup{labelformat=empty}	
\captionof{figure}{$D_0$}	
\end{minipage}
\quad\quad
\begin{minipage}{.18\textwidth}
\centering 
\begin{tikzpicture}[x=\unitsize,y=\unitsize,baseline=0]
\tikzset{vertex/.style={}}%
\tikzset{edge/.style={  thick}}%
\draw[edge] (0,2) -- (4,2);
\draw[edge] (0,4) -- (4,4);
\draw[edge] (0,6) -- (10,6);
\draw[edge] (0,8) -- (10,8);
\draw[edge] (0,10) -- (10,10);
\draw[edge] (0,2) -- (0,10);
\draw[edge] (2,2) -- (2,10);
\draw[edge] (4,2) -- (4,10);
\draw[edge] (6,6) -- (6,10);
\draw[edge] (8,6) -- (8,10);
\draw[edge] (10,6) -- (10,10);
\draw[red, line width=6pt]  (3,4) -- (3,7) -- (6,7);
\draw[red, line width=6pt]  (1,4) -- (1,9) -- (6,9);
\draw[fill=green] (3,3) circle [radius=0.5];
\draw[fill=green] (7,7) circle [radius=0.5];
\draw[fill=green] (1,3) circle [radius=0.5];
\draw[fill=green] (7,9) circle [radius=0.5];
\draw[fill=green] (9,7) circle [radius=0.5];
\draw[fill=green] (9,9) circle [radius=0.5];
\draw[dotted] (10,2) -- (10,10);
\draw[dotted] (0,2) -- (10,2);
\end{tikzpicture}
\vspace{-.5cm}
\captionsetup{labelformat=empty}	
\captionof{figure}{$D_0$}	
\end{minipage}

\end{center}

\medskip

In particular, by Theorem \ref{mainthm}, the multiplicity of $D_s$ in $H^{(m-p)(n-p)+j}_{\Z_p}(\X,\mc{O}_{\X})$ is the number of ways to remove $p-s$ hooks from $\underline{p}$ with $j$ bullets. Since said hooks must start and end on the same antidiagonal, and $m\geq n$, there will always be a $(p-s)\times (m-n)$ rectangle of bullets in rows $s+1,\cdots,p$ of $\underline{p}$ and columns $n+1,\cdots,m$ (these are the bullets depicted in the fifth columns of the patterns above). Thus, we may reduce to the case $m=n$, and assume that $\underline{p}$ and its Dyck patterns are symmetric with respect to the operation of transposing the diagram.

Let $m=n$ and let $0\leq p\leq n$. Let $\mathbb{D}\in Dyck^{\bullet}(\underline{p})$ consist of paths of length $\geq 3$ and set $\underline{s}=\underline{p}^{\mathbb{D}}$ for $0\leq s\leq p$. The union of the support of the paths in $\mathbb{D}$ will contain all boxes $(x,y)$ such that $s+1\leq x,y\leq p+1$ (except the box $(p+1,p+1)$, that lies just outside the inside corner of $\underline{p}$). Due to the covering condition for Dyck patterns, the number of bullets in row $s+i$ must be less than or equal to the number of bullets in row $s+i+1$ for $1\leq i\leq p-1$. Thus, there is a partition shape of bullets in rows $s+1,\cdots,p$, and the transpose partition of bullets in columns $s+1,\cdots,p$. 

Here is a picture when $m=n=9$, $p=4$, and $s=1$:\\ 

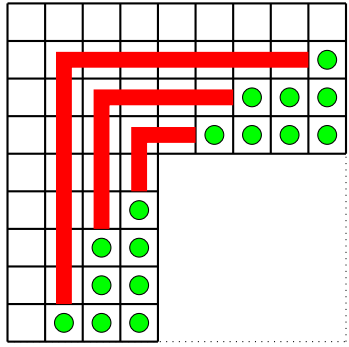
\begin{center}
\begin{minipage}{.18\textwidth}
\centering 
\begin{tikzpicture}[x=\unitsize,y=\unitsize,baseline=0]
\tikzset{vertex/.style={}}%
\tikzset{edge/.style={  thick}}%
\draw[edge] (0,18) -- (18,18);
\draw[edge] (0,16) -- (18,16);
\draw[edge] (0,14) -- (18,14);
\draw[edge] (0,12) -- (18,12);
\draw[edge] (0,10) -- (18,10);
\draw[edge] (0,8) -- (8,8);
\draw[edge] (0,6) -- (8,6);
\draw[edge] (0,4) -- (8,4);
\draw[edge] (0,2) -- (8,2);
\draw[edge] (0,0) -- (8,0);

\draw[edge] (0,0) -- (0,18);
\draw[edge] (2,0) -- (2,18);
\draw[edge] (4,0) -- (4,18);
\draw[edge] (6,0) -- (6,18);
\draw[edge] (8,0) -- (8,18);
\draw[edge] (10,10) -- (10,18);
\draw[edge] (12,10) -- (12,18);
\draw[edge] (14,10) -- (14,18);
\draw[edge] (16,10) -- (16,18);
\draw[edge] (18,10) -- (18,18);

\draw[dotted] (0,0) -- (18,0);
\draw[dotted] (18,0) -- (18,18);

\draw[red, line width=6pt] (3,2) -- (3,15) -- (16,15);
\draw[red, line width=6pt] (5,6) -- (5,13) -- (12,13);
\draw[red, line width=6pt] (7,8) -- (7,11) -- (10,11);

\draw[fill=green] (3,1) circle [radius=0.5];
\draw[fill=green] (5,1) circle [radius=0.5];
\draw[fill=green] (7,1) circle [radius=0.5];

\draw[fill=green] (5,3) circle [radius=0.5];
\draw[fill=green] (7,3) circle [radius=0.5];

\draw[fill=green] (5,5) circle [radius=0.5];
\draw[fill=green] (7,5) circle [radius=0.5];

\draw[fill=green] (7,7) circle [radius=0.5];

\draw[fill=green] (17,15) circle [radius=0.5];
\draw[fill=green] (17,13) circle [radius=0.5];
\draw[fill=green] (17,11) circle [radius=0.5];

\draw[fill=green] (15,13) circle [radius=0.5];
\draw[fill=green] (15,11) circle [radius=0.5];

\draw[fill=green] (13,13) circle [radius=0.5];
\draw[fill=green] (13,11) circle [radius=0.5];

\draw[fill=green] (11,11) circle [radius=0.5];

\end{tikzpicture}
\end{minipage}
\end{center}

\medskip

The partition formed by the bullets is $(4,3,1)$. We conclude that, for $m=n$, the multiplicity of $D_s$ in $H^{(n-p)^2+2j}_{\Z_p}(\X,\mc{O}_{\X})$ is the number of partitions of size $j$ that fit inside a $(p-s)\times (n-p-1)$ rectangle, which is the coefficient of $q^j$ in the Gaussian binomial coefficient $\binom{n-s-1}{p-s}_q$.

Putting it all together, we obtain for $m\geq n$ the following multiplicity formula:
$$
\sum_{j \geq 0}\big[ H^j_{\Z_p}(\X,\mc{O}_{\X}): D_s \big] \cdot q^j=q^{(m-p)(n-p)+(p-s)(m-n)}\cdot \binom{n-s-1}{p-s}_{q^2}.
$$
Since $(m-p)(n-p)+(p-s)(m-n)=(n-p)^2+(n-s)(m-n)$, we recover Theorem \ref{thm:detthm}(1). 

By Theorem \ref{mainthm}, the weight of $D_s=\mc{L}(\underline{p}^{\mathbb{D}})|_U$ in $H^j_{\Z_p}(\X,\mc{O}_{\X}^H)$ is $mn+|\mathbb{D}|+j=mn+p-s+j$, recovering Theorem \ref{thm:detthm}(2).

\section*{Acknowledgments}

We are grateful to Sam Evens, Andr\'{a}s L\H{o}rincz, Claudiu Raicu, and Vic Reiner for advice and useful conversations regarding this project. We thank anonymous referees for numerous helpful comments.


\begin{bibdiv}
\begin{biblist}

\bib{BB}{article}{
   title={Localisation de $\mathfrak{g}$-modules},
   author={Beilinson, A.},
   author={Bernstein, J.},
   journal={C. R. Acad. Sci. Paris S\'{e}r. I Math.},
   volume={292},
   number={1},
   year={1981},
   pages={15--18}
}

\bib{jantzen}{article}{
   author={Beilinson, A.},
   author={Bernstein, J.},
   title={A proof of Jantzen conjectures},
   conference={
      title={I. M. Gelfand Seminar},
   },
   book={
      series={Adv. Soviet Math.},
      volume={16, Part 1},
      publisher={Amer. Math. Soc., Providence, RI},
   },
   isbn={0-8218-4118-1},
   date={1993},
   pages={1--50},
}

\bib{Boe}{article}{
   author={Boe, Brian D.},
   title={Kazhdan-Lusztig polynomials for Hermitian symmetric spaces},
   journal={Trans. Amer. Math. Soc.},
   volume={309},
   date={1988},
   number={1},
   pages={279--294},
   issn={0002-9947},
}

\bib{intertwining}{article}{
   author={Boe, Brian D.},
   author={Enright, Thomas J.},
   author={Shelton, Brad},
   title={Determination of the intertwining operators for holomorphically
   induced representations of Hermitian symmetric pairs},
   journal={Pacific J. Math.},
   volume={131},
   date={1988},
   number={1},
   pages={39--50},
   issn={0030-8730},
}

\bib{beMorphisms}{article}{
 title={Determination of the interwining operators for holomorphically induced representations of $SU (p, q)$},
   author={Boe, Brian D.},
  author={Enright, Thomas J.},
  journal={Mathematische Annalen},
  volume={275},
  number={3},
  pages={401--404},
  year={1986},
  publisher={Springer}
}

\bib{borhoBry}{article}{
  title={Differential operators on homogeneous spaces. III},
  author={Borho, Walter},
  author={Brylinski, J-L},
  journal={Inventiones Mathematicae},
  volume={80},
  number={1},
  pages={1--68},
  year={1985},
  publisher={Springer}
}

\bib{brenti2}{article}{
   author={Brenti, Francesco},
   title={Kazhdan-Lusztig and $R$-polynomials, Young's lattice, and Dyck
   partitions},
   journal={Pacific J. Math.},
   volume={207},
   date={2002},
   number={2},
   pages={257--286},
   issn={0030-8730},
}


\bib{brenti}{article}{
   author={Brenti, Francesco},
   title={Parabolic Kazhdan-Lusztig polynomials for Hermitian symmetric
   pairs},
   journal={Trans. Amer. Math. Soc.},
   volume={361},
   date={2009},
   number={4},
   pages={1703--1729},
   issn={0002-9947},
}

\bib{CC}{article}{
   author={Casian, Luis G.},
   author={Collingwood, David H.},
   title={The Kazhdan-Lusztig conjecture for generalized Verma modules},
   journal={Math. Z.},
   volume={195},
   date={1987},
   number={4},
   pages={581--600},
   issn={0025-5874},
}

\bib{dualities}{book}{
   author={Cheng, Shun-Jen},
   author={Wang, Weiqiang},
   title={Dualities and representations of Lie superalgebras},
   series={Graduate Studies in Mathematics},
   volume={144},
   publisher={American Mathematical Society, Providence, RI},
   date={2012},
   pages={xviii+302},
   isbn={978-0-8218-9118-6},
}


\bib{cisWeight}{article}{
  title={Filtrations on generalized Verma modules for Hermitian symmetric pairs},
  author={Collingwood, David H.},
  author={Irving, Ronald S},
  author={Shelton, Brad},
  year={1988},
  journal={J. Reine Angew. Math.},
  volume={383},
  pages={54--86}
}

\bib{davis}{article}{
  title={Unitary representations of real groups and localization theory for Hodge modules},
  author={Davis, Dougal},
  author={Vilonen, Kari},
  journal={arXiv preprint arXiv:2309.13215},
  year={2023}
}

\bib{DEP}{article}{
   author={de Concini, Corrado},
   author={Eisenbud, David},
   author={Procesi, Claudio},
   title={Young diagrams and determinantal varieties},
   journal={Invent. Math.},
   volume={56},
   date={1980},
   number={2},
   pages={129--165},
   issn={0020-9910},
}


\bib{diagrams}{article}{
   author={Enright, Thomas J.},
   author={Hunziker, Markus},
   author={Pruett, W. Andrew},
   title={Diagrams of Hermitian type, highest weight modules, and syzygies
   of determinantal varieties},
   conference={
      title={Symmetry: representation theory and its applications},
   },
   book={
      series={Progr. Math.},
      volume={257},
      publisher={Birkh\"{a}user/Springer, New York},
   },
   date={2014},
   pages={121--184},
}



\bib{holland}{article}{
  title={K-theory of twisted differential operators on flag varieties},
  author={Holland, Martin P.},
  author={Polo, Patrick},
  journal={Inventiones Mathematicae},
  volume={123},
  number={1},
  pages={377--414},
  year={1996},
  publisher={Springer}
}

\bib{htt}{book}{
   author={Hotta, Ryoshi},
   author={Takeuchi, Kiyoshi},
   author={Tanisaki, Toshiyuki},
   title={$D$-modules, perverse sheaves, and representation theory},
   series={Progress in Mathematics},
   volume={236},
   note={Translated from the 1995 Japanese edition by Takeuchi},
   publisher={Birkh\"{a}user Boston, Inc., Boston, MA},
   date={2008},
   pages={xii+407},
}

\bib{amy}{article}{
 title={Syzygies of Determinantal Thickenings},
  author={Huang, Hang},
  journal={arXiv preprint arXiv:2008.02690},
  year={2020}
}

\bib{humphreys}{book}{
  title={Representations of Semisimple Lie Algebras in the BGG Category $\mathcal{O}$},
  author={Humphreys, James E},
  series={Graduate Studies in Mathematics},
  volume={94},
  year={2008},
  publisher={AMS}
}

\bib{irvingBook}{book}{
  title={A Filtered Category $\mathcal{O}_S$ and Applications},
  author={Irving, Ronald S.},
  volume={83, no. 419},
  number={419},
  year={1990},
  series={Memoirs of the American Mathematical Society}
}

\bib{KL}{article}{
   author={Kashiwara, Masaki},
   author={Lauritzen, Niels},
   title={Local cohomology and $D$-affinity in positive characteristic},
   language={English, with English and French summaries},
   journal={C. R. Math. Acad. Sci. Paris},
   volume={335},
   date={2002},
   number={12},
   pages={993--996},
   issn={1631-073X},
}


\bib{KT}{article}{
   author={Kashiwara, Masaki},
   author={Tanisaki, Toshiyuki},
   title={Parabolic Kazhdan-Lusztig polynomials and Schubert varieties},
   journal={J. Algebra},
   volume={249},
   date={2002},
   number={2},
   pages={306--325},
   issn={0021-8693},

}

\bib{kempf}{article}{
   author={Kempf, George},
   title={The Grothendieck-Cousin complex of an induced representation},
   journal={Adv. in Math.},
   volume={29},
   date={1978},
   number={3},
   pages={310--396},
   issn={0001-8708},
}

\bib{standard}{book}{
   author={Lakshmibai, V.},
   author={Raghavan, K. N.},
   title={Standard monomial theory},
   series={Encyclopaedia of Mathematical Sciences},
   volume={137},
   note={Invariant theoretic approach;
   Invariant Theory and Algebraic Transformation Groups, 8},
   publisher={Springer-Verlag, Berlin},
   date={2008},
   pages={xiv+265},
}

\bib{lak}{article}{
   author={Lakshmibai, V.},
   author={Weyman, J.},
   title={Multiplicities of points on a Schubert variety in a minuscule
   $G/P$},
   journal={Adv. Math.},
   volume={84},
   date={1990},
   number={2},
   pages={179--208},
   issn={0001-8708},
}

\bib{lascoux}{article}{
   author={Lascoux, Alain},
   author={Sch\"{u}tzenberger, Marcel-Paul},
   title={Polyn\^{o}mes de Kazhdan \& Lusztig pour les grassmanniennes},
   language={French},
   conference={
      title={Young tableaux and Schur functors in algebra and geometry},
      address={Toru\'{n}},
      date={1980},
   },
   book={
      series={Ast\'{e}risque},
      volume={87-88},
      publisher={Soc. Math. France, Paris},
   },
   date={1981},
   pages={249--266},
}

\bib{lauritzen}{article}{
   author={Lauritzen, Niels},
   author={Raben-Pedersen, Ulf},
   author={Thomsen, Jesper Funch},
   title={Global $F$-regularity of Schubert varieties with applications to
   $\mathcal{D}$-modules},
   journal={J. Amer. Math. Soc.},
   volume={19},
   date={2006},
   number={2},
   pages={345--355},
   issn={0894-0347},
}


\bib{iterated}{article}{
   author={L\H{o}rincz, Andr\'{a}s C.},
   author={Raicu, Claudiu},
   title={Iterated local cohomology groups and Lyubeznik numbers for
   determinantal rings},
   journal={Algebra Number Theory},
   volume={14},
   date={2020},
   number={9},
   pages={2533--2569},
   issn={1937-0652},
}

\bib{categories}{article}{
    AUTHOR = {L\H{o}rincz, Andr\'{a}s C.}
    author= {Walther, Uli},
     TITLE = {On categories of equivariant {$D$}-modules},
   JOURNAL = {Adv. Math.},
  FJOURNAL = {Advances in Mathematics},
    VOLUME = {351},
      YEAR = {2019},
     PAGES = {429--478},
      ISSN = {0001-8708},
   MRCLASS = {14F10 (14L30 14M27 16G20)},
  MRNUMBER = {3952575},
MRREVIEWER = {P. E. Newstead},
}



\bib{MP1}{article}{
    AUTHOR = {Musta\c{t}\u{a}, Mircea},
    author= {Popa, Mihnea},
     TITLE = {Hodge ideals},
   JOURNAL = {Mem. Amer. Math. Soc.},
  FJOURNAL = {Memoirs of the American Mathematical Society},
    VOLUME = {262},
      YEAR = {2019},
    NUMBER = {1268},
      ISSN = {0065-9266},
      ISBN = {978-1-4704-3781-7; 978-1-4704-5509-5},
   MRCLASS = {14D07 (14F17 14J17 32S25)},
  MRNUMBER = {4044463},
MRREVIEWER = {Matthias Wendt},
       URL = {https://doi-org.proxy.queensu.ca/10.1090/memo/1268},
}


\bib{MP4}{article}{
   author={Musta\c{t}\u{a}, Mircea},
   author={Popa, Mihnea},
   title={Hodge filtration on local cohomology, Du Bois complex and local
   cohomological dimension},
   journal={Forum Math. Pi},
   volume={10},
   date={2022},
   pages={Paper No. e22, 58},
}



\bib{MHM}{article}{
  author={Perlman, Michael},
   title={Mixed Hodge structure on local cohomology with support in
   determinantal varieties},
   journal={Int. Math. Res. Not. IMRN},
   date={2024},
   number={1},
   pages={331--358},
   issn={1073-7928},
}



\bib{perlmanraicu}{article}{
  title={Hodge ideals for the determinant hypersurface},
  author={Perlman, Michael},
  author={Raicu, Claudiu},
  journal={Selecta Mathematica New Series},
  volume={27},
  number={1},
  pages={1--22},
  year={2021},
  publisher={Springer}
}



\bib{peterssteen}{book}{
   author={Peters, Chris A. M.},
   author={Steenbrink, Joseph H. M.},
   title={Mixed Hodge structures},
   series={Ergebnisse der Mathematik und ihrer Grenzgebiete. 3. Folge. A
   Series of Modern Surveys in Mathematics},
   volume={52},
   publisher={Springer-Verlag, Berlin},
   date={2008},
   pages={xiv+470},
   }
   
 \bib{raben}{article}{
   author={Raben-Pedersen, Ulf},
   title={Local Cohomology of Schubert Varieties},
   journal={Ph.D. Thesis, Aarhus University},
   year={2005}
}

\bib{characters}{article}{
   author={Raicu, Claudiu},
   title={Characters of equivariant $\mathcal{D}$-modules on spaces of matrices},
   journal={Compos. Math.},
   volume={152},
   date={2016},
   number={9},
   pages={1935--1965},
   issn={0010-437X},
}


\bib{rw}{article}{
   author={Raicu, Claudiu},
   author={Weyman, Jerzy},
   title={Local cohomology with support in generic determinantal ideals},
   journal={Algebra Number Theory},
   volume={8},
   date={2014},
   number={5},
   pages={1231--1257},
   issn={1937-0652},
}
   
   
\bib{raicu2016local}{article}{
  title={Local cohomology with support in ideals of symmetric minors and {P}faffians},
  author={Raicu, Claudiu},
  author={Weyman, Jerzy},
  journal={Journal of the London Mathematical Society},
  volume={94},
  number={3},
  pages={709--725},
  year={2016},
  publisher={Oxford University Press}
}  

\bib{thick}{article}{
   author={Raicu, Claudiu},
   author={Weyman, Jerzy},
   title={Syzygies of determinantal thickenings and representations of the
   general linear Lie superalgebra},
   journal={Acta Math. Vietnam.},
   volume={44},
   date={2019},
   number={1},
   pages={269--284},
   issn={0251-4184},
} 

\bib{witt}{article}{
   author={Raicu, Claudiu},
   author={Weyman, Jerzy},
   author={Witt, Emily E.},
   title={Local cohomology with support in ideals of maximal minors and
   sub-maximal Pfaffians},
   journal={Adv. Math.},
   volume={250},
   date={2014},
   pages={596--610},
   issn={0001-8708},
   }


\bib{normal}{article}{
   author={Ramanan, S.},
   author={Ramanathan, A.},
   title={Projective normality of flag varieties and Schubert varieties},
   journal={Invent. Math.},
   volume={79},
   date={1985},
   number={2},
   pages={217--224},
   issn={0020-9910},
}

\bib{saito89}{article}{
   author={Saito, Morihiko},
   title={Modules de Hodge polarisables},
   language={French},
   journal={Publ. Res. Inst. Math. Sci.},
   volume={24},
   date={1988},
   number={6},
   pages={849--995 (1989)},
   issn={0034-5318},
  }


\bib{saito90}{article}{
   author={Saito, Morihiko},
   title={Mixed Hodge modules},
   journal={Publ. Res. Inst. Math. Sci.},
   volume={26},
   date={1990},
   number={2},
   pages={221--333},
   issn={0034-5318},

}


\bib{path}{article}{
  title={Path representation of maximal parabolic Kazhdan--Lusztig polynomials},
  author={Shigechi, Keiichi}, 
  author={Zinn-Justin, Paul},
  journal={Journal of Pure and Applied Algebra},
  volume={216},
  number={11},
  pages={2533--2548},
  year={2012},
  publisher={Elsevier}
}

\bib{super}{article}{
  title={Generalised Jantzen filtration of Lie superalgebras I},
  author={Su, Yucai},
  author={Zhang, Ruibin},
  journal={Journal of the European Mathematical Society},
  volume={14},
  number={4},
  pages={1103--1133},
  year={2012}
}


\end{biblist}
\end{bibdiv}
\end{document}